\documentclass[11pt,reqno]{amsart}

\usepackage{mathtools}
\usepackage{enumerate}

\newtheorem{theoremn}{Theorem}

\newtheorem{proposition}{Proposition}[section]
\newtheorem{lemma}[proposition]{Lemma}
\newtheorem{corollary}[proposition]{Corollary}
\newtheorem{theorem}[proposition]{Theorem}

\theoremstyle{definition}
\newtheorem{definition}[proposition]{Definition}
\newtheorem{example}[proposition]{Example}
\newtheorem{examples}[proposition]{Examples}
\newtheorem{remark}[proposition]{Remark}
\newtheorem{remarks}[proposition]{Remarks}

\newcommand{\thlabel}[1]{\label{th:#1}}
\newcommand{\thref}[1]{Theorem~\ref{th:#1}}
\newcommand{\selabel}[1]{\label{se:#1}}
\newcommand{\seref}[1]{Section~\ref{se:#1}}
\newcommand{\lelabel}[1]{\label{le:#1}}
\newcommand{\leref}[1]{Lemma~\ref{le:#1}}
\newcommand{\prlabel}[1]{\label{pr:#1}}
\newcommand{\prref}[1]{Proposition~\ref{pr:#1}}
\newcommand{\colabel}[1]{\label{co:#1}}
\newcommand{\coref}[1]{Corollary~\ref{co:#1}}
\newcommand{\relabel}[1]{\label{re:#1}}
\newcommand{\reref}[1]{Remark~\ref{re:#1}}
\newcommand{\exlabel}[1]{\label{ex:#1}}
\newcommand{\exref}[1]{Example~\ref{ex:#1}}
\newcommand{\delabel}[1]{\label{de:#1}}
\newcommand{\deref}[1]{Definition~\ref{de:#1}}
\newcommand{\eqlabel}[1]{\label{eq:#1}}
\newcommand{\equref}[1]{(\ref{eq:#1})}

\newcommand{\amp}{{\rm amp}}

\newcommand{\End}{{\rm End}}

\newcommand{\Aut}{{\rm Aut}\,}

\newcommand{\im}{{\rm Im}\,}

\newcommand{\id}{{\rm id}\,}

\def\NN{{\mathbb N}}

\def\ZZ{{\mathbb Z}}
\def\QQ{{\mathbb Q}}

\newcommand{\Cc}{\mathcal{C}}

\newcommand{\Oo}{\mathcal{O}}

\def\*C{{}^*\hspace*{-1pt}{\Cc}}
\def\text#1{{\rm {\rm #1}}}

\def\fse{\mathcal{FSE}}

\input xy
\xyoption {all} \CompileMatrices

\usepackage{amssymb}
\usepackage{color,amssymb,graphicx,amscd,amsmath,dsfont,stmaryrd,xcolor,comment}
\usepackage[pagebackref,colorlinks,urlcolor=blue,linkcolor=blue,citecolor=blue]{hyperref}

\usepackage{tikz}
\usetikzlibrary{arrows,arrows.meta,calc,decorations.pathreplacing,decorations.markings,intersections,shapes.geometric,through,fit,shapes.symbols,positioning,decorations.pathmorphing,backgrounds}

\numberwithin{equation}{section}

\begin{document}
\title[Set-theoretic solutions of the Frobenius-Separability equation]
{The set-theoretic Yang-Baxter equation, Kimura semigroups and functional graphs}

\author{A. L. Agore}
\address{Simion Stoilow Institute of Mathematics of the Romanian Academy, P.O. Box 1-764, 014700 Bucharest, Romania and Vrije Universiteit Brussel, Pleinlaan 2, B-1050 Brussels, Belgium}
\email{ana.agore@gmail.com and ana.agore@vub.be}

\author{A. Chirvasitu}
\address{Department of Mathematics, University at Buffalo
Buffalo, NY 14260-2900, USA}
\email{achirvas@buffalo.edu}

\author{G. Militaru}
\address{Faculty of Mathematics and Computer Science, University of Bucharest, Str. Academiei 14, RO-010014 Bucharest 1, Romania
and Simion Stoilow Institute of Mathematics of the Romanian Academy, P.O. Box 1-764, 014700 Bucharest, Romania}
\email{gigel.militaru@fmi.unibuc.ro and gigel.militaru@gmail.com}

\thanks{A.L.A. and G.M. are partially supported by a grant of the Ministry of Research, Innovation and Digitization, CNCS/CCCDI --
UEFISCDI, project number PN-III-P4-ID-PCE-2020-0458, within PNCDI III. A.C. is partially supported by NSF grant DMS-2001128.}


\maketitle

\begin{abstract}
We prove that the category of solutions of the set-theoretic Yang-Baxter equation of Frobenius-Separability (FS) type is equivalent to the category of pointed Kimura semigroups. As applications, all involutive, idempotent, nondegenerate, surjective, finite order, unitary or indecomposable solutions of FS type are classified. For instance, if $|X| = n$, then the number of isomorphism classes of all such solutions on $X$ that are (a) left non-degenerate, (b) bijective, (c) unitary or (d) indecomposable and left-nondegenerate is: (a) the Davis number $d(n)$, (b) $\sum_{m|n} \, p(m)$, where $p(m)$ is the Euler partition number, (c) $\tau(n) + \sum_{d|n}\left\lfloor \frac d2\right\rfloor$, where $\tau(n)$ is the number of divisors of $n$, or (d) the Harary number $\mathfrak{c} (n)$. The automorphism groups of such solutions can also be recovered as automorphism groups $\mathrm{Aut}(f)$ of sets $X$ equipped with a single endo-function $f:X\to X$. We describe all groups of the form $\mathrm{Aut}(f)$ as iterations of direct and (possibly infinite) wreath products of cyclic or full symmetric groups, characterize the abelian ones as products of cyclic groups, and produce examples of symmetry groups of FS solutions not of the form $\mathrm{Aut}(f)$.
\end{abstract}


\section*{Introduction}
Numerous families of equations originating in physics have received special attention in the mathematical community due to their ubiquitous and far-reaching applications. Some of the more popular, to name a few, include the Yang-Baxter equation \cite{baxter, Yang}, the pentagon equation \cite{Baaj, bieden} and the tetrahedron (or Zamolodchikov) equation \cite{zam2}. A rich variety of topics are now linked to the study of their solutions: for instance, any finite-dimensional quantum group is characterized by an invertible solution of the pentagon equation \cite{MilG}, while solutions of the tetrahedron equation have very important algebro-geometric properties \cite{BVSS, dimakis, TDV}.

The lion's share of the attention has doubtless accrued to the celebrated quantum Yang-Baxter equation, introduced independently by Yang \cite{Yang} and Baxter \cite{baxter} in their work on theoretical physics and statistical mechanics. Since then, the equation has proved to play a central role in various fields, both in mathematical physics and in pure mathematics: integrable systems, conformal field theory, quantum groups, invariants of knots and three-dimensional manifolds, tensor categories, combinatorics, and so on.

Despite volumes of work devoted to it over the past 50 years (see \cite{lamrad} and the references therein), significant classification results are still lacking in the original setting of vector spaces. To fill this gap, Drinfel'd \cite[Section 9]{dr} recasts the problem at the level of sets. In other words, for a given set $X$, Drinfel'd asked for the classification of all set-theoretic solutions of the quantum Yang-Baxter equation, i.e. all maps $R: X\times X \to X\times X$ satisfying the equation:
\begin{equation}\eqlabel{braideqint}
  R^{12}R^{23}R^{12} = R^{23}R^{12}R^{23}
\end{equation}
as maps $X\times X \times X \to X\times X \times X$ under the usual composition, where $R^{12} = R \times {\rm Id}_X$ and $R^{23} = {\rm Id}_X \times R$.\footnote{This equation is typically referred to as the \emph{braid equation} and it is equivalent to the Yang-Baxter equation as appeared in \cite{dr} in the sense that $R$ is a solution of \equref{braideqint} if and only if $\tilde{R} := \tau_X \circ R$ is a solution of $\tilde{R}^{12}\tilde{R}^{13}\tilde{R}^{23} = \tilde{R}^{23}\tilde{R}^{13}\tilde{R}^{12}$, where $\tau_X$ is the flip map.}

Even in this simplified setting the problem turns out to be very difficult. The pioneering articles \cite{etingof, GIVB, LuYZ, WeXu} have decisively influenced the development of the theory, which is now overwhelmingly concerned with non-degenerate solutions satisfying additional assumptions such as bijectivity, involutivity, unitarity, etc. (we refer to \seref{prel} for precise definitions). \emph{Bijective 1-cocycles} were first considered in \cite{EPGS} and it was proven in \cite{etingof} that the category of non-degenerate involutive solutions of the set-theoretic Yang-Baxter equation is equivalent to that of bijective $1$-cocycles with coefficients in a free abelian group \cite[Theorem 2.10]{etingof}. Several equivalent reinterpretations of this notion as well as relevant generalizations (see, e.g. \cite{brz}) have emerged and a vast literature has developed around these new algebraic gadgets. We refer to \cite[Section 3]{bai_guo} and \cite[Remark 5.6]{cm1} for a complete categorical picture of the connections between the various structures involved in studying the set-theoretic Yang-Baxter equation. Upon removing the aforementioned restrictions, the problem in Drinfel'd's original formulation is still wide open.

As opposed to the well-documented theory of non-degenerate solutions, other classes of solutions have only recently been considered (see, for instance, \cite{BKSV, Verwimp, SV, cat1}). The present paper adopts this same approach of considering solutions which, a priori, do not fulfill the classical constraints (non-degeneracy, bijectivity and so on). More specifically, we introduce a new class of solutions of the Yang-Baxter equation whose structure is completely determined and various classification results are proved using combinatorial invariants. Furthermore, the automorphism groups of these solutions are described and the groups that appear as automorphism groups of solutions from this class are identified.

To this end, we use a certain variety of semigroups\footnote{The key role played by semigroup theory in studying the Yang-Baxter equation has been highlighted for the first time in \cite{GIVB}.} introduced by Kimura \cite{Kim2} together with the functional digraphs of Harary \cite{harary} to study a new class of solutions of the Yang-Baxter equation, namely those of \emph{Frobenius-Separability (FS) type}, i.e. functions $R: X\times X \to X\times X$ satisfying the equation
\begin{equation} \eqlabel{FSprima}
R^{12}R^{23} = R^{23}R^{13} = R^{13} R^{12}.
\end{equation}

The linear version of equation \equref{FSprima} surfaced for the first time in the work of Beidar, Fong and Stolin \cite[Lemma 3.6]{BeidarFS97} on Frobenius and Hopf algebras. It was observed in \cite{BeidarFS97} that any solution of \equref{FSprima} is automatically a solution of the quantum Yang-Baxter equation. The aforementioned equation has been thoroughly investigated in \cite{CIM} and \cite[Chapter 8]{CMZ}, under the name of \emph{Frobenius-Separability equation}. The name is motivated by the main results in \cite[Theorem 4.1 and Theorem 4.3]{CIM} which state that defining a finite-dimensional Frobenius (resp. separable) algebra is equivalent to providing a solution of the equation \equref{FSprima} which moreover fulfills the normalization Frobenius (resp. separability) condition. Furthermore, classifying these two important classes of associative algebras comes down to classifying the solutions of the Frobenius-Separability equation. Unfortunately, the variety of solutions of the equation \equref{FSprima} in its linear form is extremely vast and therefore, following Drinfel'd \cite{dr}, this paper is concerned with the study of its set-theoretic solutions. Weaving together structural results proved by Kimura \cite{Kim2} on rectangular semigroups with functional graphs and their associated Davis/Harary numbers \cite{davis, harary} as well as classical number theory techniques, we classify the solutions of \equref{FSprima} and compute their automorphism groups.

The paper is organized as follows. In \seref{prel} we recall the main ingredients used throughout the paper, e.g. directed graphs associated to a function, Davis numbers and {\it Kimura semigroups}, i.e. those satisfying the identity $xyz = xz$. In this context, the structure theorems on this class of semigroups as proved by Kimura \cite[Section 2]{Kim2} will central to our approach (\thref{structura} and \cite[Theorem 2]{Kim2}).

\seref{sect2} is devoted to an in depth study of equation \equref{FSprima}; throughout, a map $R: X\times X \to X\times X$ satisfying \equref{FSprima} will be called briefly a \emph{solution on $X$}. We denote by ${\mathcal FSE}$ the category of all solutions, i.e. objects of $\mathcal{FSE}$ are pairs $(X, \, R)$, where $X$ is a set and $R$ a solution on $X$. Similarly, we denote by ${\rm KimSemigr}^{\bullet}$ the category of all \emph{pointed Kimura semigroups}: the objects are pairs $(X, \, \vartheta)$, where $(X, \, \cdot)$ is a Kimura semigroup and $\vartheta : X \to X$ is a \emph{quasi-endomorphism} of $(X, \, \cdot)$, i.e. $y \cdot \vartheta (x \cdot y) = y \cdot \vartheta (y)$, for all $x$, $y\in X$. Our first result is \thref{refqm}:

\begin{theoremn}\thlabel{refqmintro}
  There exists an equivalence of categories between the category $\mathcal{FSE}$ of all solutions and the category ${\rm KimSemigr}^{\bullet}$ of all pointed Kimura semigroups.

  Explicitly, a map $R \colon X\times X \to X\times X$ is a solution on a set $X$ if and only if there exists a Kimura semigroup structure $(X, \cdot)$ on $X$ and a quasi-endomorphism $\vartheta : X \to X$ of $(X, \cdot)$ such that $R = R_{(\cdot, \, \vartheta)}$, where $R_{(\cdot, \, \vartheta)} \colon X\times X \to X\times X$
  is given for any $x$, $y \in X$ by:
  \begin{equation*}\eqlabel{toatesolintro}
    R_{(\cdot, \, \vartheta)} (x, \, y) = \bigl( x\cdot y, \, y\cdot \vartheta (x) \bigl).
  \end{equation*}
\end{theoremn}

Moreover, relying on the general structure of Kimura semigroups as described in \cite[Theorem 2]{Kim2} we can further rephrase \thref{refqmintro} in an equivalent manner (\thref{th:refqmbis}), phrased in the combinatorial language of sets and self-maps rather than in terms of semigroups.

\begin{theoremn}\thlabel{th:refqmbisintr}
  Giving a solution $R$ on a set $X$ is equivalent to specifying a quadruple $(h, \, A, \, B, \, \vartheta')$ consisting of: (1) an idempotent self-map $h:X\to X$; (2) two non-empty sets $A$ and $B$ together with projections $\pi_1:\im~h \to A$ and $\pi_2:\im~(h) \to B$ giving an identification $\im~(h) \cong A\times B$; (3) a function $\vartheta' :B\to B$.  With these data, $R$ is defined for any $x$, $y\in X$ by:
  \begin{equation*}
    R(x, \, y) = \bigl((\pi_1 h x, \, \pi_2 h y), \, (\pi_1 h y, \, \vartheta'\pi_2 h x) \bigl).
  \end{equation*}
\end{theoremn}

The automorphism group of a given solution $(X, \, R) \in \mathcal{FSE}$ is described for both variants of the above theorems. \thref{refqmintro} has numerous applications: for instance, in \coref{invo} we classify all involutive solutions on a given set $X$ and compute their automorphism groups. In particular, if $X$ is a finite set with $|X| = n$, then the number of isomorphism classes of all involutive solutions on $X$ is equal to $\tau(n)$, the number of divisors of $n$. Furthermore, \coref{special_sol} classifies all idempotent solutions: if $X$ is a finite set with $|X| = n$, then the number of isomorphism classes of all idempotent solutions on $X$ is equal to the Euler partition number $p(n)$.
The classification of all non-degenerate solutions on $X$ as given in \thref{th:nedegenerate} can be briefly stated as follows:

\begin{theoremn}\thlabel{th:nedegenerateintro}
Let $X$ be a finite set with $|X| = n$. Then the number of isomorphism classes of all left (resp. right) non-degenerate solutions on $X$ is equal to $d(n)$, the Davis number of $n$ (resp. the Euler partition number $p(n)$).
\end{theoremn}

Next, a very large class of solutions on $X$ (which includes surjective, finite order or unitary solutions) is classified in \thref{th:main}, and their automorphism groups are described. Summarizing:

\begin{theoremn}\thlabel{th:mainintro}
  Let $R: X\times X \to X\times X$ be a solution such that $\pi_1 \circ R : X\times X \to X$ is surjective. There exist two non-empty sets $A$ and $B$, a map $\omega: B\to B$ and an isomorphism of solutions $(X, \, R) \cong (A \times B, \, R_{\omega})$, where
    $R_{\omega}: (A\times B)^2 \to (A\times B)^2$ is the solution on $A\times B$ given for any $a_1$, $a_2 \in A$ and $b_1$, $b_2 \in B$ by:
    \begin{equation*}\eqlabel{main1intro}
      R_{\omega} \bigl( (a_1, \, b_1), \, (a_2, \, b_2) \bigl) = \bigl( (a_1, \, b_2), \, (a_2, \, \omega(b_1) ) \bigl).
    \end{equation*}
\end{theoremn}
We point out that, with the exception of $A$ being a singleton, all solutions of the form $R_{\omega}$ classified above are degenerate (see (\ref{degenerate_solutions}) of \reref{singel}).

We only indicate a few of the various consequences of \thref{th:mainintro}. To this end, for a positive integer $n$ let us denote by ${\rm fs}^1 \, (n)$ (resp. ${\rm fs}_{\rm b} \, (n)$) the number of isomorphism classes of all solutions $R$ on a finite set with $n$ elements such that $\pi_1 \circ R $ is surjective (resp. $R$ is bijective). \coref{cazfinit} allows us to express ${\rm fs}^1 \, (n)$ and ${\rm fs}_{\rm b} \, (n)$ in terms of the Davis numbers and the Euler partition numbers:
\begin{eqnarray*}
{\rm fs}^1 \, (n) = \sum_{m|n} \, d(m), \qquad {\rm fs}_{\rm b} \, (n) = \sum_{m|n} \, p(m). \eqlabel{formdn0}
\end{eqnarray*}

As for unitary solutions, per \coref{cor:unitary}, if $|X| = n$ then their number of isomorphism classes is
\begin{equation*}
  \tau(n) + \sum_{d|n}\left\lfloor \frac d2\right\rfloor.
\end{equation*}

\coref{finitord} classifies finite-order solutions, using two well known combinatorial numbers: the \emph{Landau number} $g(n)$ \cite{land}, defined as the maximum order of a permutation in $S_n$, and $\overline{T (n, k)}$, the number of all conjugation classes of permutations of order $k$ in the symmetric group $S_n$ for $1 < k \leq g(n)$. If $X$ is a finite set with $|X| = n$, then the number of isomorphism classes of solutions of finite order $\geq 2$ on $X$ is equal to
\begin{equation*}
  \sum_{d|n} \, \Bigl(\sum_{k=1}^{g(d)} \, \overline{T (d, k)} \Bigl).
\end{equation*}

\seref{sect3} deals with the amply familiar notion in the Yang-Baxter literature (see the recent paper \cite{APPJ23}) of \emph{indecomposable solution} as introduced
in \cite[Definition 2.5]{etingof}. Perhaps surprisingly, and by contrast to the Yang-Baxter case, the variety of indecomposable solutions of the Frobenius-Separability equation is rather narrow, as we will see in \coref{camputine} and \coref{clasinderighnede}. For instance, any surjective solution is decomposable and on a set with $n$ elements there is, up to isomorphism, only one right non-degenerate indecomposable solution: the one associated to a cycle of length $n$.

\thref{th:conngrph} is a bit of a departure from the theme that indecomposable solutions are rather simple in character: it links indecomposability with the connectedness of the associated graph $\Gamma_f$ of a given function $f: X \to X$ as introduced in \deref{def:grph}. More precisely, \thref{th:conngrph} describes all left-nondegenerate indecomposable solutions on a set $X$ as follows:

\begin{theoremn}\thlabel{th:conngrphintro}
  Let $R = R_f$ be a left-nondegenerate solution on a set $X$ associated to a self-function $f: X \to X$. The following statements are equivalent:
  \vspace{-.1cm}
  \begin{enumerate}[(a)]
  \item \label{item:g13} $R_f$ is indecomposable;

  \item\label{item:14} The functional digraph $\Gamma_f$ is connected;

  \item\label{item:15} For every $x$, $y\in X$ there are non-negative integers $m$ and $n$ with $f^m (x) = f^n (y)$;

  \item\label{item:g15} There is no non-trivial $f$-invariant partition of $X$, i.e. $X = Y\sqcup Z$ such that $f(Y) \sqsubseteq Y$ and
    $f(Z) \sqsubseteq Z$.
  \end{enumerate}
  A function $f: X \to X$ satisfying one of the above conditions will be called \emph{connected}.
\end{theoremn}

As a consequence, \coref{clasinderighnede} shows that the number of isomorphism classes of indecomposable left-nondegenerate solutions on a finite set $X$ with $|X| = n$ is equal to the {\it Harary number} $\mathfrak{c} (n)$ introduced in \deref{hbnumber} as the number of all conjugation classes of connected self-maps $f: \{1, \cdots, n\} \to \{1, \cdots, n\}$. Equivalently, $\mathfrak{c} (n)$ is equal to the number of isomorphism classes of connected digraphs of the form $\Gamma_f$, for some $f: \{1, \cdots, n\} \to \{1, \cdots, n\}$. The number $\mathfrak{c} (n)$ appeared for the first time (albeit in disguise) in the work of Harary \cite[\S 4]{harary} and was later rediscovered as the number $p_n$ by de Bruijn \cite[p.18]{bruijn}.

\seref{sect4} focuses on the structure and classification of the symmetry groups $\Aut(X,R)$ for $\fse$ solutions $R$ on $X$. While the entire class of such groups is still somewhat elusive, we obtain partial results in a number of directions.

On the one hand, the automorphism groups of solutions of the form $R_f$ for self-functions $f:X\to X$ can be described completely. This is very much in the spirit of much other work describing, say, the automorphism groups of trees \cite[\S IV]{prins_phd} or planar graphs \cite{babai_planar-2} (typically in the context of finite structures); while the results seem to be as complete as one might hope for, they are somewhat awkward to state fully and concisely, and take up \prref{pr:autgpdec} and Theorems \ref{th:th:autogp-types}, \ref{th:th:arbprod}, \ref{th:th:allwreathcyclic}, \ref{th:th:alleveryonefix} and \ref{th:th:itwr}. A very rough summary would run approximately as follows (the symbol `$\wr$' denoting {\it wreath products} \cite[Chapter 7, preceding Theorem7.24]{rot_gp} by permutation groups):

\begin{theoremn}
  \begin{enumerate}[(1)]
  \item The groups realizable as $\Aut(X,R_f)\cong \Aut(f)$ for some endo-function $f:X\to X$ are precisely those of the form
    \begin{equation*}
      \prod_i \left(G_i\wr \Sigma_{C_i}\right),
    \end{equation*}
    where $G_i$ are automorphism groups of connected functions $f_i:X_i\to X_i$ and $\Sigma_i$ are full symmetric groups on arbitrary respective sets $C_i$.
  \item The groups realizable as $\Aut(f)$ for a {\it connected} endo-function $f:X\to X$ are precisely those falling into one of the following two classes:
    \begin{enumerate}[(a)]

    \item of the form
      \begin{equation*}
        \left(\prod_{i=1}^t G_i\right)\wr \left(\textrm{a cyclic group acting on itself by translation}\right);
      \end{equation*}

    \item countable unions of countable products of groups $G_i$, with connecting embeddings of a very special character, describable explicitly.

    \end{enumerate}
    In both of these, the $G_i$ are arbitrary automorphism groups of connected functions with (necessarily unique) fixed points.

  \item The groups realizable as $\Aut(f)$ for $f$ connected with a fixed point are the iterated (possibly infinite) wreath products of products of full symmetric groups.
  \end{enumerate}
\end{theoremn}

The reader familiar with such classification work will no doubt recognize the somewhat involved nature of the statements: see for instance \cite[Main Theorem 8.3]{babai_planar-2}.

For an easily-stated consequence, \coref{cor:abautf} classifies the abelian groups $\Aut(f)$ (and in fact the abelianizations $\Aut(f)_{ab}$) as precisely the products of (finite or infinite) cyclic groups. The restrictions such groups are subject to are thus significant. Contrast this with \thref{th:allgpsquotkimura}, which shows that the class of automorphism groups of Kimura semigroups (and hence, by extension, that of groups $\Aut(X,R)$ for $\fse$ solutions $R$) is rather broad:

\begin{theoremn}
  An arbitrary group can be realized as a split quotient of $\Aut(X,\cdot)$ for some Kimura semigroup $(X,\cdot)$, with kernel isomorphic to a product of full permutation groups.
\end{theoremn}

This is sufficient (\coref{cor:notallrf}) to give examples of groups of the form
\begin{equation*}
  \Aut(X,\cdot),\quad (X,\cdot)\textrm{ a Kimura semigroup}
\end{equation*}
(hence also of the form $\Aut(X,R)$ by \thref{refqmintro}) that are {\it not} realizable as $\Aut(Y,R_f)$ for any endo-function $f:Y\to Y$.

\section{Preliminaries}\selabel{prel}
All sets in this paper will be non-empty, ${\rm Id}_X$ is the identity map on a set $X$ and $\{\star\}$ denotes the singleton set.  The group of permutations on a set $X$ will be denoted by $\Sigma_X$ or by $S_n$ if $X$ is a finite set with $n$ elements.

\begin{definition}\delabel{def:grupuriciud}
  \begin{enumerate}[(1)]
  \item The {\it centralizer} or {\it automorphism group} of a tuple $(f_i)_i$ of functions $f_i:X\to X$, $i\in I$ on a set $X$ is
    \begin{equation*}
      \Aut(f_i,\ i\in I):=\{\sigma \in \Sigma_X \, | \, \sigma \circ f_i = f_i \circ \sigma,\ \forall i\in I\}\leq \Sigma_X .
    \end{equation*}
    In the special case of one function $f: X\to X$ we simply write {\it function centralizer} and denote by $\Aut(f)$.
  \item For a cardinal number $\alpha$, an abstract group is an {\it $\alpha$-function centralizer} if it is isomorphic to the centralizer of a family $(f_i)_{i\in \alpha}$ of $\alpha$ endomorphisms on some set $X$.
  \end{enumerate}
\end{definition}

In the course of working with self-functions $f:X\to X$ it will be convenient to have available some graph-theoretic tools and terminology. We introduce the following objects both for this purpose, and for future reference.

\begin{definition}\delabel{def:grph}
  For a function $f:X\to X$, the {\it directed graph} $\Gamma_f$ {\it associated (or attached) to $f$} is the directed graph with
  \begin{itemize}
  \item $X$ as its vertex set;
  \item and an edge $x\to y$ precisely when $y = f (x)$ and $x\ne y$.
  \end{itemize}
  In other words, this is the usual graph $\{(x, \, f(x) )\ |\ x\in X\}$ minus the diagonal $\Delta_X:=\{(x, x)\ |\ x\in X\}$.
\end{definition}

\begin{remark}\relabel{re:fngrphs}
  The graphs of the form $\Gamma_f$ are (almost) the {\it functional digraphs} of \cite[\S 2]{harary}: that paper's digraphs disallow length-1 loops $x\to x$ by definition, and are termed functional when every vertex has out-degree 1. This means that Harary's functional digraphs are those $\Gamma_f$ for $f:X\to X$ having no fixed points. We will apply the term 'functional (di)graph' more broadly, to all $\Gamma_f$ (for $f$ possibly having fixed points).

  The directed graphs $\Gamma_f$ and $\Gamma_g$ associated to two functions $f$, $g: X \to X$ will be called \emph{isomorphic}, and we write $\Gamma_f \cong \Gamma_g$, if they are isomorphic in the usual sense \cite[pag. 3]{Die00}: there exists a permutation $\sigma : X \to X$ such that $x\to y$ is an edge in $\Gamma_f$ if and only if $\sigma(x)\to \sigma(y)$ is an edge in $\Gamma_g$. We can easily prove that $\Gamma_f \cong \Gamma_g$ if and only if $f$ and $g$ are conjugate maps, i.e. there exists $\sigma \in \Sigma_X$ such that $g = \sigma \circ f \circ \sigma^{-1}$. In particular, we obtain that the automorphism group $\Aut (\Gamma_f ) \cong \Aut(f)$, the function centralizer.
\end{remark}

The group $\Sigma_X$ acts by conjugation on itself: the cardinal number of all conjugation classes of permutations on $X$ is denoted by $p(X)$: if $|X| = n$, then $p(X) = p(n)$, the number of all possible partitions of the positive integer $n$ as introduced and studied for the first time by Euler \cite{euler}. Since then, a vast literature devoted to the study of $p(n)$ was developed in number theory and combinatorics; for more details, including generating functions and asymptotic formulas for computing $p(n)$, we refer to \cite{andrews}. On the other hand, the same group $\Sigma_X$ acts by conjugation on the set $X^X$ of all functions $f: X \to X$, i.e. $\sigma \rhd f := \sigma \circ f \circ \sigma^{-1}$, for all $\sigma \in \Sigma_X$ and $f:X \to X$.  The cardinal number of all conjugation classes of maps $X\to X$ is denoted by $d(X)$ and it will be called the \emph{Davis number} of $X$ (in combinatorics it is also known as \emph{mapping patterns} or \emph{mapping types} of $X$); if $X$ is a finite set with $|X| = n$, then $d(X)$ will be denoted by $d(n)$.

\begin{remark}\relabel{davids}
  Using \reref{re:fngrphs} we obtain that the number $d(n)$ (sequence \cite{slo-A001372}) is also that of isomorphism classes of functional graphs $\Gamma_f$ on $n$ vertices. It was introduced by Davis in \cite[Theorem 6]{davis} where the first formula for its computation was given. Independently, de Bruijn \cite[page 18]{bruijn} in his study of mapping types of a set gave a more explicit formula for computing $d(n)$, namely:
  \begin{equation} \eqlabel{bruijn}
    d(n) = \sum_{(m)} \, \prod_{i=1}^{\infty} \, \Bigl(\sum_{j|i}  \, j \, m_j \Bigl)^{m_i} \, \frac{1}{i^{m_i} m_i!}
  \end{equation}
  where the first summation is over all sequences $(m) = (m_1, \, m_2, \, \cdots )$ such that $m_1 + 2\, m_2 + 3\, m_3 + \cdots = n$. Based on this formula, the first fifteen values were computed: $d (1) = 1, \cdots, d(15) = 1328993$. Moreover, a generating function and connections with rooted trees have been highlighted; for more detail see \cite{harary, read}. We just mention that, if $T_k$ is the number of rooted trees with $k$ points (sequence \cite{slo-A000081}) and
  \begin{equation*}
    T(x) = \sum_{k\geq 1} \, T_k \, x^k
  \end{equation*}
  its generating function (also denoted by $T(x)$ on \cite[p.19]{bruijn}), then the generating function for $d(n)$ was given by Read \cite{read}:
  \begin{equation} \eqlabel{davids2}
    \sum_{n\geq 1} \, d(n) x^n = \prod_{n=1}^{\infty} \, \frac{1}{1-T(x^n)}.
  \end{equation}
\end{remark}

A semigroup is a set $S$ with an associative binary operation $m_S: S\times S \to S$, $m_S (x, y) = xy$, for all $x$, $y\in S$ called multiplication. For the basics of semigroup theory we refer to \cite{Grillet, Howie}. We only recall briefly the concepts that will be used throughout the paper. A semigroup $S$ is called \emph{total} if $m_S$ is surjective, that is if $S = S^2 := \{xy \, | \, x, y \in S \}$.
For any set $S$, by defining the multiplication $x y : = x$ (resp. $x y : = y$), for all $x$, $y\in S$ we obtain two semigroups called the \emph{left zero} (resp. \emph{right zero}) \emph{semigroup} on $S$. The set of idempotents of a semigroup $S$ is denoted by $E(S) = \{ e \in S \, | \, e = e^2\}$ and $S$ is called a \emph{band} (or an \emph{idempotent semigroup}) if $E(S) = S$. A semigroup $S$ is called \emph{rectangular} if
\begin{equation}\label{rectang}
xyx = x
\end{equation}
for all $x$, $y\in S$. Any rectangular semigroup $S$ is a band \cite{Kim2} and the basic example of a rectangular semigroup is the so-called
\emph{rectangular band} of two sets $A$ and $B$, i.e. $S := A\times B$ with the multiplication given for any $a_1$, $a_2 \in A$ and $b_1$, $b_2 \in B$ by:
\begin{equation} \eqlabel{recband}
(a_1, \, b_1) \cdot (a_2, \, b_2) := (a_1, \, b_2).
\end{equation}

The following class of semigroups, introduced by Kimura \cite{Kim2} as a proper generalization of rectangular semigroups, will play a key role in this paper:

\begin{definition}\delabel{kimura}
A \emph{Kimura semigroup} is a semigroup $S$ such that for any $x$, $y$, $z\in S$:
\begin{equation}\eqlabel{kim}
xyz = xz.
\end{equation}
\end{definition}

Kimura semigroups have been generalized to \emph{exclusive semigroups} which are semigroups $S$ satisfying the (ambiguous) compatibility condition: $x y z \in \{xy, \, yz, \, xz \}$, for all $x$, $y$, $z\in Z$. Structural results for exclusive semigroups were obtained by Yamada \cite{yamada}.

\begin{examples}\exlabel{exkimura}
  \;
  \begin{enumerate}[(1)]
  \item It can be easily seen that if a Kimura semigroup $S$ has a unit, then $S$ is the singleton $S = \{\star\}$. If a Kimura semigroup $S$ has a zero element $0$ (i.e. an element $0\in S$ such that $x 0 = 0 x = 0$, for all $x\in S$), then $S$ is the trivial semigroup, i.e. $x y = 0$, for all $x$, $y\in S$.
  
      Left/right zero semigroups are Kimura semigroups. If a Kimura semigroup $S$ is a left (resp. right) cancellative semigroup, then $S$ is a right (resp. left) zero semigroup.

  \item Any rectangular semigroup is a Kimura semigroup \cite{Kim2}. In particular, the rectangular band of two sets $A$ and $B$ defined in \equref{recband} is a Kimura semigroup.

  \item\label{item:5} Let $S$ be a set and $f = f^2 : S \to S$ an idempotent map, i.e. $f(f(x)) = f(x)$, for all $x\in S$. We denote ${}_fS : = S$, as a set, with the multiplication $x \cdot y := f(x)$, for all $x$, $y\in S$. Then ${}_fS$ is a Kimura semigroup. In a similar fashion, $S_f := S$ with multiplication $x \cdot y := f(y)$, for all $x$, $y\in S$, is a Kimura semigroup. Two semigroups ${}_fS$ and ${}_gS$ (resp. $S_f$ and $S_g$) are isomorphic if and only if $f$ and $g$ are conjugate idempotents, i.e. there exists $\sigma \colon S \to S$ a permutation on the set $S$ such that $g = \sigma \circ f \circ \sigma^{-1}$. Moreover, the automorphism group $${\rm Aut} ({}_fS ) = {\rm Aut} (S_f ) = \Aut(f) = \{ \sigma \in \Sigma_S \, | \, \sigma \circ f = f \circ \sigma \}$$ the function centralizer of $f$. We can also prove that ${}_fS$ and $S_g$ are isomorphic if and only if $f$ and $g$ are constant maps, i.e. there exist $c_1$, $c_2 \in S$ such that $f(x) = c_1$ and $g(x) = c_2$, for all $x\in S$. If $f\neq {\rm Id}_X$, then both ${}_fS$ and $S_f$ are examples of Kimura semigroups which are not rectangular.

  \item The direct product of two Kimura semigroups is a Kimura semigroup. In particular, for two idempotent maps $f= f^2$ and $g=g^2$ on a set $S$, we obtain that ${}_fS \times S_g$ is a Kimura semigroup which is not isomorphic to a Kimura semigroup of the form ${}_h T$ or $T_h$, for some set $T$ and idempotent map $h$.

  \item Let $S$ and $T$ be two Kimura semigroups and $\varphi : T \to {\rm End}\, (S)$ a morphism of semigroups, where ${\rm End}\, (S)$ denotes the monoid of all endomorphisms of $S$.  We denote $\varphi (t) (s) = t \vartriangleright s$ and consider $S \rtimes_{\varphi} \, T$ to be the semidirect product of semigroups,
      i.e. $S \rtimes_{\varphi} \, T$ is the semigroup with underlying set $S\times T$ and multiplication given for any $s$, $s'\in S$ and $t$, $t'\in T$ by:
      \begin{equation*}(s, \, t) \cdot (s', \, t') := \bigl( s (t \vartriangleright s'), \, tt'\bigl).\end{equation*}
      It can be easily seen that $S \rtimes_{\varphi} \, T$ is a Kimura semigroup if and only if $s \bigl( (tt') \vartriangleright s' \bigl) = s (t\vartriangleright s')$,
      for all $s$, $s'\in S$ and $t$, $t'\in T$.

      \item\label{item:afine} Let $S$ be a cancellative semigroup and $A$, $B \in {\rm End}(S)$ two endomorphisms of $S$. We define a new multiplication on $S$ given by $x \cdot y := A(x) B(y)$, for all $x$, $y\in S$. Then, it can be easily seen that $(S, \cdot)$ is a Kimura semigroup if and only if there exist two elements $s_1$, $s_2 \in S$ such that for any $x \in S$:
       \begin{eqnarray*}
A (B (x))  &=& s_1, \qquad A(x) = A^2 (x) \, s_1 \\
B (A (x))  &=& s_2, \qquad B(x) = s_2 \,  B^2 (x)
\end{eqnarray*}

  \item\label{item:finitary} Note that the category ${\rm KimSemigr}$ of all Kimura semigroups is a {\it finitary variety of algebras} (\cite[p.124]{mcl} or \cite[Definition 16.16]{ahs}), i.e. a category of sets equipped with a number of {\it finitary} operations (each having a finite number of arguments), satisfying certain equations. It follows, in particular \cite[pp.124-125]{mcl}, that the forgetful functor
    \begin{equation*}
      {\rm KimSemigr} \xrightarrow[]{\mathrm{forget}} {\rm Set}
    \end{equation*}
    to the category of sets has a left adjoint $\mathcal{FK} (-) : {\rm Set} \to {\rm KimSemigr}$, where for any set $X$, $\mathcal{FK} (X)$ is the free Kimura semigroup generated by $X$.

    The structure of $\mathcal{FK}(X)$ is very simple: as a set it consists exactly of the 1- or 2-letter words on the alphabet $X$, with multiplication defined by
    \begin{equation*}
      (x\cdots)\cdot (\cdots x') = xx';
    \end{equation*}
    or: concatenate the two words and remove all but the first and last letter.

  \end{enumerate}
\end{examples}

The connection between rectangular and Kimura semigroups was established by the following structure theorem proved for the most part in \cite[Section 2]{Kim2} (see also \cite[Theorem 1.1.3]{Howie}). Recall also that an {\it ideal} \cite[p.4]{Howie} in a semigroup is a subset closed under both left and right multiplication by arbitrary elements, and that a semigroup is
\begin{itemize}
\item {\it simple} \cite[\S 3.1]{Howie} if it has no proper ideals;
\item and {\it completely simple} \cite[\S 3.3]{Howie} if it is simple and has a {\it primitive idempotent} \cite[\S 3.2]{Howie}, i.e. an idempotent element $e$ such that
  \begin{equation*}
    f\textrm{ idempotent and }fe=ef=f\Longrightarrow f=e.
  \end{equation*}
\end{itemize}

\begin{theorem}\thlabel{structura}
  Let $S$ be a semigroup. The following statements are equivalent:
  \begin{enumerate}[(1)]
  \item $S$ is a rectangular semigroup;
  \item $S$ is a band and a Kimura semigroup;
  \item $S$ is both a total and a Kimura semigroup;
  \item There exist two non-empty sets $A$ and $B$ and an isomorphism of semigroups $S \cong A \times B$, where $A\times B$ is the rectangular band of the sets $A$ and $B$ as defined by \equref{recband};
  \item There exist a left zero semigroup $A$, a right zero semigroup $B$ and an isomorphism of semigroups $S \cong A \times B$. Moreover, this factorization is unique up to an isomorphism of semigroups.
  \item\label{Kim6} $S$ is Kimura and simple;
  \item\label{Kim7} $S$ is Kimura and completely simple.
  \end{enumerate}
\end{theorem}

\begin{remarks}\relabel{singletonAB}
  \;
  \begin{enumerate}
    \item Statements (\ref{Kim6}) and (\ref{Kim7}) of \thref{structura} do not appear in the above cited papers \cite{Howie, Kim2}, but they are easily proven equivalent to the others. Indeed, if $S$ is a simple Kimura semigroup and $S^2$ is obviously an ideal of $S$ we arrive at $S^{2}= S$ which shows that $S$ is isomorphic to a rectangular band. Conversely, any rectangular band is a (completely) simple Kimura semigroup.

  \item For future use, we indicate the construction of the non-empty sets $A$ and $B$ of the above theorem as given in \cite[Lemma 1]{Kim2}. If $S$ is a rectangular semigroup, then $A$ (resp. $B$) is the set of all subsets of $S$ of the form $xS$ (resp. $Sx$), for all $x\in S$. Then $A$ (resp. $B$) is a non-empty set and $(xS) (yS) = xS$ (resp. $(Sx)(Sy) = Sy$), for all $x$, $y\in S$, i.e. $A$ (resp. $B$) is a left zero (resp. right zero) semigroup and $r : S \to A\times B$, $r(x) := (xS, \, Sx)$, for all $x\in S$ is an isomorphism of semigroups. Furthermore, we can easily see that in the above decomposition (unique up to an isomorphism of semigroups) we have $A = \{\star\}$ if and only if there exists an idempotent map $f = f^2 \colon S \to S$ such that $S = S_f$ and $B = \{\star\}$ if and only if there exists $f = f^2 \colon S \to S$ such that $S = {}_fS$.

  \item A structure theorem for Kimura semigroups was proven in \cite[Theorem 2]{Kim2}: for any Kimura semigroup $S$ there exist a rectangular subsemigroup $R \subseteq S$ of $S$ and a partition
    \begin{equation*}
      S = \coprod_{r\in R}S_r\textrm{ with }r\in S_r\textrm{ and }S_rS_t = \{rt\},\ \forall r,t\in R.
    \end{equation*}
  \end{enumerate}
\end{remarks}

\thref{structura} has an immediate consequence: for a prime number $p$, there are only two isomorphism types of rectangular semigroups of order $p$.

\begin{corollary}\colabel{primrect}
  Let $p$ be a prime number and $S$ a rectangular semigroup with $|S| = p$. Then $S$ is isomorphic to either the left zero semigroup or the right zero semigroup on $S$.
\end{corollary}

The following result, used extensively below, is probably folklore:

\begin{lemma}\lelabel{morfrecta}
  Let $A$, $B$, $C$ and $D$ be four sets. Then there exists a bijection between the set of all morphisms of rectangular bands $\psi : A\times B \to C\times D$ and the set of all pairs $(\sigma_1, \, \sigma_2)$, where $\sigma_1: A\to C$, $\sigma_2: B\to D$ are two maps. Under the above bijection the semigroup morphism $\psi = \psi_{(\sigma_1, \, \sigma_2)}: A\times B \to C\times D$ corresponding to $(\sigma_1, \, \sigma_2)$ is given by $\psi = \sigma_1 \times \sigma_2$.  Furthermore, there exists an isomorphism of groups ${\rm Aut} (A\times B) \cong \Sigma_A \times \Sigma_B$, where ${\rm Aut} (A\times B)$ is the automorphism group of the rectangular band $A\times B$.
\end{lemma}

\begin{proof}
Let $\psi : A\times B \to C\times D$ be a morphism of semigroups and $(a_0, \, b_0) \in A\times B$ a fixed element. Now define $\sigma_1: A\to C$ and $\sigma_2: B\to D$
by the formulas:
\begin{equation*}
\sigma_1 (a) := (\pi_C \circ \psi) (a, \, b_0), \qquad
\sigma_2 (b) := (\pi_D \circ \psi) (a_0, \, b)
\end{equation*}
for all $a\in A$ and $b\in B$, where $\pi_C\colon C\times D \to C$ and $\pi_D\colon C\times D \to D$ denote the canonical projections. We will show that $\sigma_1\colon A\to C$ and $\sigma_2 \colon B\to D$ do not depend on the choice of the element $(a_0, \, b_0) \in A\times B$. Indeed, let $(x, \, y) \in A\times B$ be an arbitrary element and $a\in A$. Using \equref{recband} we obtain:
\begin{eqnarray*}
\sigma_1 (a) &=& (\pi_C \circ \psi) (a, \, b_0) \\
&=& (\pi_C \circ \psi) \bigl( (a, \, y) \cdot (a_0, \, b_0) \bigl)\\
&=& \pi_C \bigl( \psi (a, \, y) \cdot  \psi (a_0, \, b_0) \bigl) \\
&=& \pi_C \bigl( (\pi_C \circ \psi) (a, \, y), \,  (\pi_D \circ \psi) (a_0, \, b_0) \bigl) \\
&=& (\pi_C \circ \psi) (a, \, y)
\end{eqnarray*}
Thus, $\sigma_1 (a) = (\pi_C \circ \psi) (a, \, y)$, for all $y\in B$. In a similar fashion, we can prove that
$\sigma_2 (b) = (\pi_D \circ \psi) (x, \, b)$, for all $x\in A$ and this finishes the proof.
\end{proof}

\section{Set-theoretic Frobenius-Separability equation}\selabel{sect2}

We introduce the set-theoretic version of the following equation which first appeared in \cite[Lemma 3.6]{BeidarFS97}. Its linear version was thoroughly studied in \cite{CIM} and \cite[Chapter 8]{CMZ} in connection to two classes of associative algebras: separable and Frobenius algebras.

\begin{definition} \delabel{fseq}
Let $X$ be a set. A map $R: X\times X \to X\times X$ is called a solution of the \emph{set-theoretic Frobenius-Separability equation} if
\begin{equation}\eqlabel{FS}
R^{12}R^{23} = R^{23}R^{13} = R^{13}R^{12}
\end{equation}
as maps $X\times X \times X \to X\times X \times X$ under the usual composition of maps. From now on, a map $R$ satisfying equation \equref{FS} will be called briefly a \emph{solution on $X$} and will be denoted by $R(x,y) = (x \cdot y, \, x\ast y) = \bigl(l_x(y), \, r_y(x) \bigl)$, for all $x$, $y\in X$.
\end{definition}

For a given set $X$, we denote by $\pi_1$, $\pi_2 \colon X\times X \to X$ the canonical projections on the first/second component, i.e. $\pi_1 (x, y) = x$, $\pi_2 (x, y) = y$, for all $x$, $y\in X$. For two sets $X$ and $Y$ we denote by $\tau_{X, Y} \colon X\times Y \to Y \times X$, the flip map defined by $\tau_{X, Y} \, (x, y) = (y, x)$, for all $x\in X$ and $y\in Y$. A map $R\colon X \times X \to X\times X$ will be written as:
\begin{equation}\eqlabel{notatR}
R(x, \, y) = (x \cdot y, \, x\ast y) = \bigl(l_x(y), \, r_y(x) \bigl), \quad {\rm for\,\, all\,\,} x,\, y \, \in X
\end{equation}
where $x \cdot y = l_x(y) := (\pi_1 \circ R) (x,y)$ and $x \ast y = r_y(x) := (\pi_2 \circ R) (x,y)$. For a map $R: X \times X \to X\times X$ we denote:
\begin{equation*}
  R^{12} := R \times {\rm Id}_X, \quad  R^{23} := {\rm Id}_X \times R, \quad R^{13} : X \times X \times X \to X \times X \times X
\end{equation*}
where $R^{13} = ({\rm Id}_X \times \tau_X) \circ R^{12} \circ ({\rm Id}_X \times \tau_X)$ and $\tau_X = \tau_{X, X}$ is the flip map.
The map $R \colon X \times X \to X\times X$ will be called:
\begin{enumerate}
\item \emph{left non-degenerate} (resp. \emph{right non-degenerate}) if the maps $l_x \colon X \to X$ (resp. $r_x : X \to X$) are bijections for all $x\in X$. Furthermore, $R$ is called non-degenerate if it is both left and right non-degenerate \cite[Definition 1.1]{etingof};

\item \emph{idempotent} if $R^2 = R$;

\item \emph{involutive} (resp. of \emph{finite order}) if $R^2 = {\rm Id}_{X\times X}$ (resp. $R^n = {\rm Id}_{X\times X}$, for some positive integer $n$);

\item \emph{unitary} \cite{dr} if $R^{21}R = {\rm Id}_{X\times X}$, where $R^{21} := \tau_X R \tau_X$;

\item \emph{diagonal} if $R(x, \,x) = (x, \,x)$, for all $x \in X$;

\item \emph{commutative} (resp. \emph{cocommutative}) \cite{Baaj} if $R^{12}R^{13} = R^{13}R^{12}$ (resp. $R^{13}R^{23} = R^{23}R^{13}$).
\end{enumerate}

First of all, remark that any solution $R$ on $X$ is a solution of the set-theoretic Yang-Baxter equation on $X$. Indeed, using \equref{FS} repeatedly yields:
\begin{equation*}
  R^{12}R^{23}R^{12} = R^{23}R^{13}R^{12} = R^{23}R^{12}R^{23}
\end{equation*}

Furthermore, in \reref{puterisolutii} we will see that if $R$ is a solution on $X$ and $n$ is a positive integer, then
$\tau_X \circ R^{2n}$ is also a set-theoretic solution of the Yang-Baxter equation on $X$.\\

Let $\mathcal{FSE}$ be the category of solutions of \equref{FS}, i.e. objects of $\mathcal{FSE}$ are pairs $(X, \, R)$, where $X$ is a set and $R$ a solution on $X$. A morphism $\sigma: (X, \, R) \to (X', \, R')$ in $\mathcal{FSE}$ is a map $\sigma: X \to X'$ such that $(\sigma \times \sigma) \circ R = R' \circ (\sigma \times \sigma)$. This is equivalent to $\sigma: X \to X'$ being a map such that $\sigma (x \cdot y) = \sigma(x) \cdot' \sigma(y)$ and $\sigma (x \ast y) = \sigma(x) \ast' \sigma(y)$, for all $x$, $y\in X$. Two solutions $(X, \, R)$ and $(X', \, R') \in \mathcal{FSE}$ are called \emph{isomorphic} if they are isomorphic as objects in the category $\mathcal{FSE}$. For a given solution $(X, \, R) \in \mathcal{FSE}$ we denote by ${\rm Aut} \, (X, \, R)$ its automorphism group in the category $\mathcal{FSE}$.

\begin{remark}\relabel{re:refelctive}
  The subcategory $\mathcal{FSE}$ is \emph{reflective} \cite[\S IV.3]{mcl} in the category $\mathcal{YB}$ of all solutions of the Yang-Baxter equation: the inclusion functor $i: \mathcal{FSE} \lhook\joinrel\xrightarrow{\quad} \mathcal{YB}$ has a left adjoint. This follows immediately from Freyd's celebrated {\it adjoint functor theorem} \cite[Corollary 8.17, p. 28]{bergman} applied to varieties of algebras. For more details see \cite[Section 5]{cm1}.
\end{remark}

\begin{examples}\exlabel{exemFS}
  \;
  \begin{enumerate}[(1)]

  \item For any set $X$, the identity map ${\rm Id}_{X\times X}$ and the flip map $\tau_{X} : X\times X \to X \times X$, $\tau_{X} \, (x, \, y) = (y, \, x)$ are solutions on $X$. 
  
  \item \label{item:constsol} If $a\in X$ is a fixed element of a set $X$, then $R_a (x, \, y) := (a, \, a)$, for all $x$, $y\in X$ is a solution on $X$ called the \emph{constant solution} on $X$. Any two constant solutions $R_a$ and $R_b$ are isomorphic and $\Aut (X, \, R_a) = \{ \sigma \in \Sigma_X  \, | \, \sigma (a) = a \} \cong \Sigma_{X \setminus \{a\}}$. We will show in \exref{simetrice} that the constant solutions are the only \emph{symmetric solutions} on a set $X$, that is solutions satisfying $R(x, \, y) = R(y, \, x)$, for all $x$, $y\in X$

  \item\label{item:fh} Let $f$, $h: X \to X$ be two maps and $R = R_{(h, \, f)}$ defined for any $x$, $y\in X$ by:
    \begin{equation}\eqlabel{exfsih}
      R_{(h, \, f)} : X\times X \to X\times X, \qquad R_{(h, \, f)} \, (x, \, y) := (h(y), \, f(x)).
    \end{equation}
    Then $R_{(h, \, f)}$ is a solution on $X$ if and only if $h^2 = h$ and $f = f\circ h = h \circ f$. Indeed, for any $x$, $y$ and $z\in X$ we have:
    \begin{eqnarray*}
      && R^{12}R^{23} (x, y, z) = \bigl(h^2(z), \, f(x), \, f(y) \bigl) \\
      && R^{23}R^{13} (x, y, z) = \bigl(h(z), \, h(f(x)), \, f(y)) \bigl)    \\
      && R^{13}R^{12} (x, y, z) = \bigl(h(z), \, f(x), \, f(h(y)) \bigl)
    \end{eqnarray*}
    and the conclusion follows. Furthermore, we can easily see that:
    \begin{equation} \eqlabel{centdoua}
      {\rm Aut} \, (X, \, R_{(h, \, f)})  = \Aut(h, \, f ) =
      \{ \sigma \in \Sigma_X \, | \,\, \sigma \circ h = h \circ \sigma, \,\,\, \sigma \circ f = f \circ \sigma \}
    \end{equation}
    the centralizer of the tuple $(X, \, (h, \, f))$. In particular, $R_f :=  R_{({\rm Id}_X, \, f)}$ is a solution on $X$, for any map $f: X \to X$. We will show in \thref{th:nedegenerate} that the maps  $R_f$ are the only left non-degenerate solutions on a set $X$.

    In the same fashion, we can prove that
    \begin{equation}\eqlabel{exfsih2}
      R^{(h, \, f)} : X\times X \to X\times X, \qquad R^{(h, \, f)} \, (x, \, y) := (h(x), \, f(y))
    \end{equation}
    is a solution on $X$ if and only if $h = f = f^2$. We denote $R^f := R^{(f, \, f)}$, for an idempotent map $f: X\to X$.  We will show in \coref{special_sol}
    that the maps  $R^f$ are the only idempotent solutions on a set $X$.

  \item\label{item:weirdsolfo} Let $A$, $B$ be two sets and $f: A \to A$, $\omega : B\to B$ two maps. Let
    ${}_fR_{\omega}: (A\times B)^2 \to (A\times B)^2$ given by:
    \begin{equation*}
    {}_fR_{\omega} \bigl( (a_1, \, b_1), \, (a_2, \, b_2) \bigl) \, = \, \bigl( ( f(a_1), \, b_2), \, (f(a_2), \, \omega(b_1) ) \bigl)
    \end{equation*}
    for all $a_1$, $a_2 \in A$ and $b_1$, $b_2 \in B$. Then we can prove that ${}_fR_{\omega}$ is a solution on $A\times B$ if and only if $f^2 = f$. \thref{th:main} will prove that $R: X\times X \to X\times X$ is a solution on $X$ such that $\pi_1 \circ R : X\times X \to X$ is surjective if and only if $(X, \, R) \cong (A\times B, \, {}_{{\rm Id}_A} R_{\omega})$.

  \item\label{item:prodsol} As in \exref{exkimura} (\ref{item:finitary}), the category $\mathcal{FSE}$ is a finitary variety of algebras, so one constructs products there in the obvious fashion (the forgetful functor to ${\rm Set}$ {\it creates} products \cite[\S V.1, Definition]{mcl}). Explicitly, if $(X, \, R_X)$ and $(Y, \, R_Y) \in \mathcal{FSE}$ are two solutions, then $(X\times Y, \, R_{X} \times R_{Y} := R_Y^{24}R_X^{13} ) \in \mathcal{FSE}$ is a solution on $X \times Y$ called the \emph{direct product solution}, where $R^{13}_{X}$, $R^{24}_{Y} \colon (X\times Y) \times (X\times Y) \to (X\times Y) \times (X\times Y)$ are the maps acting as $R_X$ (resp. $R_Y$) on the first and the third (resp. the second and the fourth) factors of the Cartesian product $(X\times Y)^2$. In particular, if $(X, \, R) \in \mathcal{FSE}$ is a solution and $Y$ an arbitrary set, then $(X\times Y, \, R^{13}) \in \mathcal{FSE}$ is a solution on $X\times Y$.

  Furthermore, the product with a fixed object $(X,\,R)$ of $\fse$ gives a functor
    \begin{equation*}
      \fse\xrightarrow[]{\quad(X,\,R)\times -\quad}\fse\downarrow (X,\,R).
   \end{equation*}
    to the {\it comma category} \cite[Exercise 3K]{ahs} consisting of arrows in $\fse$ with target $(X,\,R)$. It is an immediate check that this functor is right adjoint to
   \begin{equation*}
   \fse\downarrow (X,\,R)\xrightarrow[]{\quad\mathrm{forget}\quad} \fse.
   \end{equation*}
   This is a very general observation, applicable to arbitrary categories $\Cc$ with products: the functor $c\times -$ from $\Cc\to \Cc\downarrow c$ has the forgetful functor as its left adjoint.
  \end{enumerate}
\end{examples}

Kimura semigroups arise naturally in the study of the set-theoretic Frobenius-Separability equation:

\begin{proposition}\prlabel{detFS}
  Let $X$ be a set. A map $R: X\times X \to X\times X$, $R(x, \, y) = (x \cdot y, \, x\ast y)$ is a solution on $X$ if and only if for any  $x$, $y$, $z\in X$:
  \begin{eqnarray}
    && x\cdot (y \cdot z) = (x\cdot y) \cdot z = x\cdot z, \,\, {\rm i.e. \, (X, \cdot) \,\, is \,\, a \,\, Kimura \,\, semigroup;} \eqlabel{24} \\
    && x \ast (y \cdot z) = y \cdot (x \ast z) = x \ast y;  \eqlabel{25} \\
    && (x \cdot y) \ast z = y \ast (x \ast z) = y \ast z.   \eqlabel{26}
  \end{eqnarray}
\end{proposition}

\begin{proof}
Let $x$, $y$, $z\in X$. Then a direct calculation shows that:
\begin{eqnarray*}
&& R^{12}R^{23} (x, y, z) = \bigl( x \cdot (y \cdot z), \, x\ast (y\cdot z), \, y\ast z \bigl) \\
&& R^{23}R^{13} (x, y, z) = \bigl( x\cdot z, \, y \cdot (x \ast z), \, y \ast(x \ast z) \bigl)    \\
&& R^{13}R^{12} (x, y, z) = \bigl( (x \cdot y) \cdot z, \, x \ast y, \, (x\cdot y) \ast z \bigl)
\end{eqnarray*}
Thus, $R$ is a solution on $X$ if and only if $(X, \cdot)$ is a Kimura semigroup and the compatibility conditions \equref{25} and \equref{26} hold.
\end{proof}

\begin{remark} \relabel{famimaps}
If we write $R$ as $R(x,y) = \bigl(x\cdot y , \, r_y(x) \bigl)$ we obtain that $R$ is a solution on $X$ if and only if $(X, \cdot)$ is a Kimura semigroup
and the family of maps $r_x : R \to R$, $x\in X$, satisfy the following compatibility conditions for all $y$ and $z\in X$:
\begin{eqnarray}
r_y  = r_{y\cdot z} = l_y \circ r_z = r_y \circ l_z = r_{r_{y}(z)} \eqlabel{25a}
\end{eqnarray}
where $l_x (y) = x\cdot y$, for all $x$, $y\in X$.
\end{remark}

Based on \prref{detFS} we can now write solutions on $X$ in an equivalent form. For this purpose we introduce the following concept:

\begin{definition}\delabel{qmorf}
Let $(X, \cdot)$ be a Kimura semigroup. A \emph{quasi-endomorphism} of $X$ is a map $\vartheta : X \to X$ such that for any $x$, $y \in X$ we have:
\begin{equation}\eqlabel{qmobis}
y \cdot \vartheta (x \cdot y) = y \cdot \vartheta (y).
\end{equation}
\end{definition}

If $\vartheta : X \to X$ is a quasi-endomorphism of a Kimura semigroup $(X, \cdot)$ then the following holds for all $x$, $y$, $z \in X$:
\begin{equation}\eqlabel{qmo}
z \cdot \vartheta (x \cdot y) = z \cdot \vartheta (y)
\end{equation}
Indeed, using \equref{kim} and \equref{qmobis}, we have: $z \cdot \vartheta (x \cdot y) = z \cdot y \cdot \vartheta (x \cdot y) =
z \cdot y \cdot \vartheta (y) =  z \cdot \vartheta (y)$, as needed.

\begin{examples}\exlabel{quasi-m}
  \;
  \begin{enumerate}[(1)]
  \item Any endomorphism of a Kimura semigroup $(X, \cdot)$ is also a quasi-endomorphism. Moreover, if $\vartheta : X \to X$ is a quasi-endomorphism, then if we multiply with $\vartheta (x)$ on the left of \equref{qmo} we obtain, taking into account \equref{24}, that for all $x$, $y\in X$ we have $\vartheta (x) \cdot \vartheta (y) = \vartheta (x) \cdot \vartheta (x \cdot y)$. These two observations motivate our terminology.

  \item Any ''left $X$-module map'' $\vartheta : X \to X$, i.e. a map satisfying $\vartheta (x \cdot y) = x \cdot \vartheta (y)$, for all $x$, $y\in X$, is a quasi-endomorphism of $X$. This follows trivially by using \equref{24}.

  \item\label{item:qK} Let $\vartheta$ be a quasi-morphism on a Kimura semigroup $(X, \cdot)$ such that
    \begin{equation}\eqlabel{power_sol}
      x \cdot \vartheta^{2}(y) = x \cdot \vartheta(y)
    \end{equation}
    for all $x$, $y \in X$. For example, any idempotent quasi-morphism on $(X, \cdot)$ trivially fulfills the previous condition. Then $(X, \star)$ is a Kimura semigroup and $\vartheta$ is also a quasi-morphism on $(X, \star)$, where $x \star y = x \cdot \vartheta(y)$ for all $x$, $y \in X$. Indeed, we have:
    \begin{eqnarray*}
      && (x \star y) \star z = \bigl(x \cdot \vartheta(y)\bigl) \star z = \bigl(x \cdot \vartheta(y)\bigl) \cdot \vartheta(z) \stackrel{\equref{24}} = x \cdot \vartheta(z) = x \star z\\
      && x \star (y \star z) = x \star \bigl(y \cdot \vartheta(z)\bigl) = x \cdot \vartheta\bigl(y \cdot \vartheta(z)\bigl) \stackrel{\equref{qmo}} =
         x \cdot \vartheta^{2}(z) \stackrel{\equref{power_sol}} = x \cdot \vartheta(z) = x \star z
    \end{eqnarray*}
    for all $x$, $y$, $z \in X$. This shows that $(X, \star)$ is a Kimura semigroup. Moreover, we have:
    \begin{eqnarray*}
      z \star \vartheta(x \star y) = z \cdot \vartheta^{2}\bigl(x \cdot \vartheta(y)\bigl) \stackrel{\equref{power_sol}} = z \cdot \vartheta\bigl(x \cdot \vartheta(y)\bigl) \stackrel{\equref{qmo}} = z \cdot \vartheta^{2}(y) = z \star \vartheta (y)
    \end{eqnarray*}
    for all $x$, $y$, $z \in X$, which proves our second claim.
  \end{enumerate}
\end{examples}

\begin{definition}\delabel{def:ptkim}
  The category ${\rm KimSemigr}^{\bullet}$ of \emph{pointed Kimura semigroups} has
  \begin{itemize}
  \item pairs $(X, \, \vartheta)$ as objects, where $X$ is a Kimura semigroup and $\vartheta$ is a quasi-endomorphism of $X$;
  \item as morphisms $\alpha : (X, \, \vartheta_X) \to (Y, \, \vartheta_Y)$, semigroup morphisms $\alpha : X \to Y$ such that
    \begin{equation*}
      \alpha \bigl(y \cdot_X \vartheta_X (x) \bigl) = \alpha (y) \cdot_Y \vartheta_Y (\alpha(x)),\ \forall x,y\in X.
    \end{equation*}
  \end{itemize}
\end{definition}

Quasi-endomorphisms of Kimura semigroups play a crucial role in the description of all solutions of the Frobenius-Separability equation. Indeed, we have the following theoretic description of all solutions on $X$ using only the language of semigroup theory:

\begin{theorem}\thlabel{refqm}
Let $X$ be a set. A map $R \colon X\times X \to X\times X$, $R(x, \, y) = (x \cdot y, \, x\ast y)$ is a solution on $X$ if and only if $(X, \cdot)$ is a Kimura semigroup and
there exists a quasi-endomorphism $\vartheta : X \to X$ of $(X, \cdot)$ such that $R = R_{(\cdot, \, \vartheta)}$, where the canonical map $R_{(\cdot, \, \vartheta)} \colon X\times X \to X\times X$ is defined for any $x$, $y \in X$ by:
\begin{equation}\eqlabel{toatesol}
R_{(\cdot, \, \vartheta)} (x, \, y) = \bigl( x\cdot y, \, y\cdot \vartheta (x) \bigl).
\end{equation}

The functor $F : {\rm KimSemigr}^{\bullet} \to \mathcal{FSE}$, $F (X, \, \vartheta) := (X, \, R_{(\cdot, \, \vartheta)})$, $F(\alpha) : = \alpha$,
for all $(X, \, \vartheta) \in {\rm KimSemigr}^{\bullet}$ and all morphisms $\alpha \in {\rm KimSemigr}^{\bullet}$, is an equivalence of categories.
\end{theorem}

\begin{proof}
  According to \prref{detFS}, a map $R: X\times X \to X\times X$ as depicted in the statement is a solution on $X$ if and only if $(X, \cdot)$ is a Kimura semigroup such that \equref{25} and \equref{26} hold. Let $z_0 \in X$ be an element of $X$ and define $\vartheta : X \to X$, $\vartheta (x) = x \ast z_0$, for all $x\in X$. Then the second equality in \equref{25} gives $x\ast y = y \cdot \vartheta (x)$, for all $x$, $y\in X$. As a consequence of $(X, \cdot)$ being a Kimura semigroup, \equref{25} holds and moreover, we obtain that \equref{26} holds if and only if $\vartheta : X \to X$ satisfies the compatibility condition \equref{qmo}, i.e. as we prove above, $\vartheta$ is a quasi-endomorphism of $(X, \, \cdot)$. This proves our first statement. In fact, we have proved that the functor $F : {\rm KimSemigr}^{\bullet} \to \mathcal{FSE}$, is essentially surjective. Furthermore, it can be immediately seen that $F$ is faithful while the way we defined morphisms in the category ${\rm KimSemigr}^{\bullet}$ guarantees that $F$ is also a full functor;  hence applying \cite[Theorem 1, pg. 93]{mcl} leads to the conclusion that $F$ is an equivalence of categories, as desired.
\end{proof}

In the context of \thref{refqm} we can also easily see that two solutions $(X, \, R_{(\cdot, \, \vartheta)} )$ and $(X', \, R_{(\cdot', \, \vartheta')}' ) \in \mathcal{FSE}$ are isomorphic if and only if there exists $\sigma \colon (X, \, \cdot) \to (X', \, \cdot') $ an isomorphism of Kimura semigroups such that $x' \cdot' \sigma \bigl( \vartheta(x) \bigl) = x' \cdot' \vartheta' \bigl( \sigma (x) \bigl)$, for all $x\in X$ and $x' \in X'$. In particular, the automorphism group ${\rm Aut} \, (X, \, R_{(\cdot, \, \vartheta)} )$ of the solution $(X, \, R_{(\cdot, \, \vartheta)} )$ is the subgroup of ${\rm Aut} \, (X, \, \cdot )$ consisting of all automorphisms $\sigma : X\to X$ of the Kimura semigroup $(X, \, \cdot)$ satisfying the following compatibility condition for all $x$, $y\in X$:
\begin{equation}\eqlabel{autogene}
y \cdot \sigma \bigl( \vartheta(x) \bigl) = y \cdot \vartheta \bigl( \sigma (x) \bigl).
\end{equation}

Given the relatively user-friendly structure of Kimura semigroups, we can rephrase \thref{refqm} so as to render it more amenable to computations and providing examples:

\begin{theorem}\thlabel{th:refqmbis}
  Let $X$ be a set.
  \begin{enumerate}[(1)]

  \item\label{item:solmeansthis} Giving an $\fse$ solution $R$ on $X$ is equivalent to specifying
    \begin{itemize}

    \item an idempotent function $h:X\to X$;

    \item projections
      \begin{equation*}
        \pi_1:\im~h \to A
        \quad\textrm{and}\quad
        \pi_2:\im~(h) \to B
      \end{equation*}
      identifying $\im~(h) \cong A\times B$;

    \item and a function $\vartheta':B\to B$.

    \end{itemize}
    Given these data, $R$ is given for any $x$, $y\in X$ by
    \begin{equation}\eqlabel{eq:rvartheta'}
      R(x, \, y) = ((\pi_1 h x, \, \pi_2 h y),\ (\pi_1 h y, \, \vartheta'\pi_2 h x)).
    \end{equation}

  \item\label{item:automeansthis} The automorphisms of a solution $R$ as above consist of those permutations $\sigma\in \Sigma_X$
    \begin{itemize}
    \item which centralize $h$;
    \item and induce permutations on $A$ and $B$ via $\pi_i$ respectively;
    \item and such that the induced permutation on $B$ centralizes $\vartheta'$.
    \end{itemize}

  \end{enumerate}
\end{theorem}

\begin{proof}
  This is what \thref{refqm} translates to upon unwinding the structure of a general Kimura semigroup described in \cite[Theorem 2]{Kim2} (and its proof). We tackle the two claims in turn.
  \begin{enumerate}[(1)]
  \item The idempotent function $h:X\to X$ attached to a Kimura semigroup structure thereon is simply
    \begin{equation*}
      X\ni x\xmapsto{\quad}x^2\in X,
    \end{equation*}
    and its image is the total Kimura subsemigroup $X^2\subseteq X$. Being total, that subsemigroup is canonically a rectangular band $X^2\cong A\times B$ (\thref{structura}). This disposes of $h$ and $\pi_i$.

    The other ingredient, $\vartheta'$, is an avatar of the quasi-endomorphism $\vartheta$ of \thref{refqm} (as the notation suggests): the condition \equref{qmo} of being a quasi-endomorphism means precisely that $\vartheta(xy)$ and $\vartheta(y)$ have the same image through $\pi_2 h$ (for all $x,y\in X$), and hence
      \begin{equation*}
        \pi_2 h \vartheta(xy) = \pi_2 h\vartheta(y),\ \forall x,y\in X
      \end{equation*}
      is equivalent to
      \begin{equation*}
        \pi_2 h (y') = \pi_2 h (y)
        \Longrightarrow
        \pi_2 h \vartheta (y') = \pi_2 h \vartheta (y).
      \end{equation*}
      This means that $\vartheta$ descends to a function $\vartheta':B\to B$ via
      \begin{equation*}
        \vartheta'\pi_2 h(x) := \pi_2 h\vartheta(x);
      \end{equation*}
      \equref{toatesol} can then be rewritten as \equref{eq:rvartheta'} using the newly-defined function $\vartheta'$, and this makes it clear that $R$ uniquely determines and depends on $\vartheta'$ alone (rather than $\vartheta$).

  \item The first two bullet items in claim (\ref{item:automeansthis}) are the defining properties of being a Kimura semigroup automorphism. As for the third bullet point, the compatibility condition \equref{autogene} on an automorphism $\sigma$ means precisely that
    \begin{equation*}
      \pi_2h\sigma\vartheta(x) = \pi_2h\vartheta\sigma(x),\ \forall x\in X.
    \end{equation*}
    Given that both $\vartheta$ and $\sigma$ descend to $B$ along $\pi_2$ to $\vartheta'$ and, say, $\sigma'$ respectively, this condition means that $\sigma'\vartheta'=\vartheta'\sigma'$ as functions $B\to B$.  \qedhere

  \end{enumerate}
\end{proof}

\begin{remark}\relabel{re:exrevisited}
  \thref{th:refqmbis} can help place the solution examples above in some perspective. \exref{exemFS} (\ref{item:fh}), for instance, is what \thref{th:refqmbis} specializes to when $A$ is a singleton, so that $\pi_i$ are irrelevant and $B=\im~h$. On the other hand, \exref{exemFS} (\ref{item:weirdsolfo}) is what one obtains when the idempotent function $h:X\to X$ is of the form
  \begin{equation*}
    X=A\times B\xrightarrow{\quad f\times\id_B\quad} A\times B=X.
  \end{equation*}
 \end{remark}

As a first application of \thref{refqm} we classify all involutive (resp. diagonal) solutions on a given set $X$.

\begin{corollary}\colabel{invo}
  Let $X$ be a set. Then:
  \begin{enumerate}[(1)]
  \item A solution on $X$ is involutive if and only if it is diagonal. Furthermore, $R$ is an involutive solution on $X$ if and only if $(X, \cdot)$ is a rectangular semigroup and $R$ is defined for any $x$, $y \in X$ by:
    \begin{equation}\eqlabel{idemp_sol}
      R(x, \, y) = \bigl( x\cdot y, \, y\cdot x \bigl).
    \end{equation}
  \item For a decomposition of $X = A\times B$ as a product of two non-empty sets, a system of representatives for the isomorphism classes of all
  involutive solutions on $X$ is given by the maps $R : (A\times B)^2 \to (A\times B)^2$ defined as follows:
    \begin{equation} \eqlabel{involsol11}
      R \bigl( (a_1, \, b_1), \, (a_2, \, b_2) \bigl) = \bigl( (a_1, \, b_2), \, (a_2, \, b_1 ) \bigl)
    \end{equation}
  for all $a_1$, $a_2 \in A$ and $b_1$, $b_2 \in B$. For such an $R$, there exists an isomorphism of groups $\Aut (X, \, R) \cong \Sigma_A \times \Sigma_B$.
  \end{enumerate}
  In particular, if $X$ is a finite set with $|X| = n$, then the number of isomorphism classes of all involutive solutions on $X$ is equal to $\tau(n)$,
  the number of divisors of $n$.
\end{corollary}

\begin{proof}
  \begin{enumerate}[(1)]
  \item In light of \thref{refqm}, any solution $R$ on $X$ is given by $R = R_{(\cdot, \, \vartheta)}$ as defined in \equref{toatesol}, where $(X, \, \cdot)$ is a Kimura semigroup and $\vartheta \colon X \to X$ a quasi-endomorphism of $(X, \, \cdot)$. Therefore, for any $x$, $y\in X$ we have:
    \begin{eqnarray*}
      && R^{2} (x, \, y) = R(x\cdot y, \, y\cdot \vartheta (x)) = (x \cdot y^{2} \cdot \vartheta(x),\, y \cdot \vartheta (x) \cdot \vartheta (x \cdot y)) \\
      && \stackrel{\equref{kim}} = (x  \cdot \vartheta(x), \, y  \cdot \vartheta(x \cdot y)) \stackrel{\equref{qmo}} = (x  \cdot \vartheta(x),\, y  \cdot \vartheta(y)).
    \end{eqnarray*}
    This shows that $R$ is an involutive solution if and only if
    \begin{equation}\eqlabel{idemp1}
      x  \cdot \vartheta(x) = x
    \end{equation}
    for all $x\in X$. On the other hand, $R$ is diagonal (that is $R(x, \, x) = (x, \, x)$, for all $x\in X$) if and only if $x^2 = x$ and $x \cdot  \vartheta (x) = x$, for all $x\in X$. We show now that the first condition (i.e. $(X, \cdot)$ being a band) follows from the second one; indeed, for any $x\in X$ we have:
    \begin{eqnarray*}
      x^{2} = x\cdot x = x \cdot x  \cdot \vartheta(x) \stackrel{\equref{kim}} = x  \cdot \vartheta(x) = x.
    \end{eqnarray*}
    Thus we have proved that $R$ is a diagonal solution on $X$ if and only if $R$ is involutive and both statements are equivalent
    (using $(1) \Leftrightarrow (2)$ of \thref{structura}) with the fact that $(X, \cdot)$ is a rectangular semigroup and \equref{idemp1} holds for
    a quasi-morphism $\vartheta$ of $(X, \cdot)$. Furthermore, in this case we obtain:
    \begin{eqnarray*}
      y  \cdot \vartheta(x) \stackrel{\equref{kim}} = y \cdot x  \cdot \vartheta(x) = y \cdot x
    \end{eqnarray*}
    for all $x$, $y \in X$ and therefore any involutive (diagonal) solution on $X$ takes the form depicted in \equref{idemp_sol}.

  \item We observe first that two solutions of the form \equref{idemp_sol} associated to two rectangular semigroups $(X, \cdot)$ and $(X, \cdot')$, are isomorphic if and only if the semigroups $(X, \cdot)$ and $(X, \cdot')$ are isomorphic. The conclusion follows by first applying $(1) \Leftrightarrow (4)$ of \thref{structura} and then \leref{morfrecta}. Furthermore, note that the formula \equref{involsol11} is obtained from \equref{idemp_sol} after the identification $X \cong A\times B$ and taking into account the multiplication of the rectangular band $A\times B$ given by \equref{recband}. Moreover, \leref{morfrecta} gives an isomorphism of groups $\Aut (X, \, R) \cong \Sigma_A \times \Sigma_B$ \qedhere
  \end{enumerate}
\end{proof}

The proof of \coref{invo} shows that if $R = R_{(\cdot, \, \vartheta)}$ is a solution on a set $X$, then for any $x$, $y\in X$ we have
$R^{2} (x, \, y) = (x  \cdot \vartheta(x), \, y  \cdot \vartheta(y))$, which is a map of the form \equref{exfsih2}. As $\vartheta$ is a quasi-endomorphism, we easily obtain that $R^2$ is also a solution on $X$ if and only if the following holds for all $x \in X$:

\begin{equation}\eqlabel{power_solbb}
   x \cdot \vartheta^{2}(x) = x \cdot \vartheta(x).
\end{equation}

Now we take one step forward and examine the situation when all powers $R^n$ of a solution $R$ are again solutions on $X$, where $R^{n}$ denotes the $n$-times composition of $R$ with itself. We have the following rigidity type result:

 \begin{corollary}\colabel{puteri}
  Let $R = R_{(\cdot, \, \vartheta)}$ be a solution on a set $X$. The following are equivalent:
  \begin{enumerate}[(1)]
  \item $R^n$ is a solution on $X$, for any positive integer $n$;

  \item $R^3$ is a solution on $X$;

  \item The quasi-endomorphism $\vartheta : X \to X$ satisfies the following compatibility condition:
  \begin{equation}\eqlabel{power_solcc}
   x \cdot \vartheta^{2}(y) = x \cdot \vartheta(y)
   \end{equation}
  for all $x$, $y \in X$.
\end{enumerate}
Furthermore, in this case $R^{2n} = R^2$ and $R^{2n+1} = R^3$, for all positive integers $n$.
\end{corollary}

\begin{proof} By induction we can immediately show that for any positive integer $n$ we have:
  \begin{eqnarray} \eqlabel{puterile}
    R^{n} ( x , \, y ) = \left \{\begin{array}{lll}
                                   \bigl( x \cdot \vartheta^k (x), \, y \cdot \vartheta^k (y)  \bigl),\,\,\, \mbox {if}\,\, n = 2k\\
                                   \bigl( x \cdot \vartheta^k (y), \, y \cdot \vartheta^{k+1} (x)  \bigl),\,\,\, \mbox {if} \,\, n = 2k+1
                                 \end{array} \right.
  \end{eqnarray}
  for all $x$, $y\in X$. The implication $(1) \Rightarrow (2)$ follows trivially; we show now that $(2) \Leftrightarrow (3)$. To this end, we denote $S (x, \, y) := R^3 (x, \, y) = (x \cdot \vartheta(y), \, y \cdot \vartheta^{2}(x) )$, for all $x$, $y\in X$. Using intensively that
  $(X, \cdot)$ is a Kimuara semigroup and $\vartheta : X \to X$ is a quasi-endomorphism we can prove that:
  \begin{eqnarray*}
    && S^{12}S^{23} (x, y, z) = \bigl( x \cdot \vartheta^2(z) , \, y \cdot \vartheta^2(x), \, z \cdot \vartheta^2 (y) \bigl) \\
    && S^{23}S^{13} (x, y, z) = \bigl( x \cdot \vartheta(z) , \, y \cdot \vartheta^3(x), \, z \cdot \vartheta^2 (y) \bigl)    \\
    && S^{13}S^{12} (x, y, z) = \bigl( x \cdot \vartheta(z) , \, y \cdot \vartheta^2(x), \, z \cdot \vartheta^2 (x\cdot \vartheta (y) ) \bigl)
  \end{eqnarray*}
  Hence, $R^3$ is a solution on $X$ if and only if for any $x$, $y$, $z\in X$ we have
  \begin{equation*}
    x \cdot \vartheta^2(z) = x \cdot \vartheta(z), \quad y \cdot \vartheta^2(x) = y \cdot \vartheta^3(x), \quad z \cdot \vartheta^2 (y) = z \cdot \vartheta^2 (x\cdot \vartheta (y))
  \end{equation*}
  As $\vartheta : X \to X$ is a quasi-endomorphism of the Kimura semigroup $(X, \cdot)$, both the second and third compatibilities follow from the first one, which is exactly \equref{power_solcc} and this proves $(2) \Leftrightarrow (3)$. The proof will be finished once we show that $(3) \Rightarrow (1)$. Henceforth we assume \equref{power_solcc} to hold which trivially implies that \equref{power_solbb} is fulfilled as well and therefore $R^2$ is also a solution on $X$. Furthermore, \equref{puterile} implies that for all positive integers $n$ we have $R^{2n} = R^2$ and $R^{2n+1} = R^3$ and this finishes the proof.
\end{proof}

\begin{remark} \relabel{puterisolutii}
Let $R = R_{(\cdot, \, \vartheta)}$ be a solution on a set $X$ and $n$ a positive integer. Then using \equref{puterile} we obtain that
$\tilde{R} := \tau_X \circ R^{2n}$ is a solution of the set-theoretic Yang-Baxter equation $\tilde{R}^{12}\tilde{R}^{23}\tilde{R}^{12} = \tilde{R}^{23}\tilde{R}^{12}\tilde{R}^{23}$ of Lyubashenko type \cite[Example 1]{dr}.
\end{remark}

Next we classify all idempotent FS-solutions on a given set $X$. In the process of doing so, we will run across idempotent functions $f:X\to X$, regarded as equivalent whenever they are mutually conjugate under permutations of $X$. This makes the following simple observation relevant.

\begin{lemma}\lelabel{le:idm=part}
  Let $X$ be a set. There is a bijection between the set of $\Sigma_X$-conjugacy classes of idempotent functions $f:X\to X$ and the set of partitions of the cardinal number $|X|$, meaning multisets of non-zero cardinal numbers with sum $|X|$.
\end{lemma}

\begin{proof}
  Simply note that for an idempotent self-map $f:X\to X$ the only $\Sigma_X$-invariant is the multiset of non-zero cardinal numbers
  \begin{equation*}
    |f^{-1}(x)|,\ x\in X.
  \end{equation*}
  Conversely, given such a multiset, an $f$ can be constructed which recovers it, and that $f$ will be unique up to $\Sigma_X$-conjugation.
\end{proof}

\begin{corollary}\colabel{special_sol}
 Let $X$ be a set and $R$ a solution on $X$. The following are equivalent:
  \begin{enumerate}[(1)]
  \item\label{item:1}  $R$ is idempotent;
  \item\label{item:2}  $R$ is commutative;
  \item\label{item:3}  $R$ is cocommutative;
  \item\label{item:4}  $R = R^f$, for some idempotent function $f = f^2 : X \to X$, where $R^{f}$  is given by \equref{exfsih2}, i.e. $R^f (x, \, y) = (f(x), \, f(y))$, for all $x$, $y\in X$.
  \end{enumerate}

In particular, if $X$ is a finite set with $|X| = n$, then the number of isomorphism classes of all idempotent solutions on $X$ is equal to the partition number $p(n)$, while the number of all idempotent solutions on $X$ is $\Sigma_{k=0}^n \, \tbinom{n}{k} \, k^{n-k}$.
\end{corollary}

\begin{proof}
We will show, using \thref{refqm}, that $R$ is an idempotent (resp. (co)commutative) solution on $X$ if and only if $R = R^f$, for some idempotent function $f = f^2 : X \to X$, where $R^{f}$ is given by \equref{exfsih2}.

Indeed, applying \thref{refqm} any solution $R$ on $X$ is given by $R = R_{(\cdot, \, \vartheta)}$ as defined in \equref{toatesol}, where $(X, \, \cdot)$ is a Kimura semigroup and $\vartheta \colon X \to X$ is a quasi-endomorphism of $(X, \, \cdot)$. Therefore, using \equref{24} and \equref{qmo}, we obtain as in the proof of \coref{invo}, that:
  \begin{equation*}
    R^2 (x, \, y) = \bigl(x \cdot \vartheta (x), \, y \cdot \vartheta (y) \bigl)
  \end{equation*}
Hence $R$ is idempotent if and only if $x \cdot \vartheta (x) = x\cdot y$ and $y \cdot \vartheta (y) = y \cdot \vartheta (x)$, for all $x$, $y\in X$. Setting $y = x$ and $x = z$ in this last relation we obtain
\begin{equation*}
  R^2=R
  \iff
  x \cdot \vartheta (x) = x\cdot y = x \cdot \vartheta (z)
\end{equation*}
for all $x$, $y$, $z\in X$. After fixing some $z_0 \in X$ and defining $f: X \to X$, $f (x) := x \cdot \vartheta (z_0)$ we obtain that $x \cdot y = f(x)$, for all $x$, $y\in X$ and hence $R = R^f$. Moreover, $f = f^2$ since $R$ is a solution on $X$ (see \equref{exfsih2} of \exref{exemFS}). This shows the equivalence (\ref{item:1}) $\iff$ (\ref{item:4}).

Furthermore, let $R = R_{(\cdot, \, \vartheta)}$ be a solution on $X$ as defined in \equref{toatesol}, where $(X, \, \cdot)$ is a Kimura semigroup and $\vartheta \colon X \to X$ is a quasi-endomorphism of $(X, \, \cdot)$. A straightforward computation shows that for all $x$, $y$, $z \in X$ we have:
\begin{eqnarray*}
  R^{13}R^{23}(x,\,y,\,z) &=& R^{13}(x,\, y \cdot z,\, z \cdot \vartheta(y)) = (x \cdot z \cdot \vartheta(y),\, y \cdot z,\, z \cdot \vartheta(y) \cdot \vartheta(x))\\
                                      &=& (x \cdot \vartheta(y),\, y \cdot z,\, z \cdot \vartheta(x)) \\
  R^{23}R^{13}(x,\,y,\,z) &=&  R^{23}(x \cdot z,\, y,\, z \cdot \vartheta(x)) = (x \cdot z,\, y \cdot z \cdot \vartheta(x),\, z \cdot \vartheta(x) \cdot \vartheta(y))\\
                                      &=& (x \cdot z,\, y \cdot \vartheta(x),\, z \cdot \vartheta(y))
\end{eqnarray*}
Therefore, $R$ is cocommutative if and only if we have:
\begin{equation*}x \cdot  \vartheta(y) = x \cdot z, \,\,\,y \cdot z = y \cdot \vartheta(x),\,\,\, z \cdot \vartheta(x) = z \cdot \vartheta(y)\end{equation*}
for all $x$, $y$, $z \in X$. Now if we fix $y := z_0 \in X$ and define $f: X \to X$, $f (x) := x \cdot \vartheta (z_0)$ we obtain $x \cdot z = f(x)$ and $z \cdot \vartheta(x) = f(z)$ for all $x$, $z\in X$. This shows that $R(x,\, z) = (f(x),\, f(z)) = R^{f}(x,\,z)$, for all $x$, $z\in X$. Finally, $f^{2} = f$ since $R$ is a solution on $X$ (see \equref{exfsih2} of \exref{exemFS}), which proves the equivalence (\ref{item:3}) $\iff$ (\ref{item:4}).

Similarly, for all $x$, $y$, $z \in X$ we have:
\begin{eqnarray*}
R^{12}R^{13}(x,\,y,\,z) &=& (x \cdot y,\, y \cdot \vartheta(z),\, z \cdot \vartheta(x)) \\
R^{13}R^{12}(x,\,y,\,z) &=& (x \cdot z,\, y \cdot \vartheta(x),\, z \cdot \vartheta(y))
\end{eqnarray*}
Thus, $R$ is commutative if and only if we have:
\begin{equation*}x \cdot y = x \cdot z,\,\,\, y \cdot \vartheta(z) = y \cdot \vartheta(x),\,\,\,z \cdot \vartheta(x) = z \cdot \vartheta(y)\end{equation*}
for all  $x$, $y$, $z \in X$. Now if we fix $z := z_0 \in X$ and define $f: X \to X$, $f (x) := x \cdot z_0$ we obtain $x \cdot y = f(x)$, for all $x$, $y\in X$.
Hence, we obtain that $R = R^f$ and obviously $f = f^{2}$ since $R$ is a solution on $X$ (see \equref{exfsih2} of \exref{exemFS}). The equivalence (\ref{item:2}) $\iff$ (\ref{item:4}) now follows.

Finally, we observe that two idempotent solutions $(X, \, R^f)$ and $(X, \, R^g)$ are isomorphic in $\mathcal{FSE}$ if and only if the idempotent maps $f$ and $g$ are conjugate. Thus, a system of representatives for the isomorphism classes of all idempotent solutions on $X$ is parameterized by the set of conjugation classes $[f]$ of all idempotent maps $f: X \to X$. The conjugacy statement on partition numbers then follows from \leref{le:idm=part}, while the last assertion was proven in \cite[Theorem 1]{harris}.
\end{proof}

As a last application of \thref{refqm} we classify all symmetric solutions on a set $X$.

\begin{example} \exlabel{simetrice}
Let $R$ be a symmetric solution on $X$. In light of \thref{refqm}, $R$ is given in $R = R_{(\cdot, \, \vartheta)}$ as defined by \equref{toatesol}, where $(X, \, \cdot)$ is a Kimura semigroup and $\vartheta \colon X \to X$ a quasi-endomorphism of $(X, \, \cdot)$. Such a solution is symmetric if and only if $(X, \, \cdot)$ is a commutative semigroup and $x \cdot \vartheta (y) = y \cdot \vartheta (x)$, for all $x$, $y\in X$. Considering an arbitrary element  $z \in X$ and multiplying the previous compatibility on the right with $z$ we obtain, taking into account \equref{kim}, that $x \cdot z = y \cdot z$, for all $x$, $y$, $z\in X$. Fix $y_0 \in X$ and define $\lambda : X \to X$, $\lambda (z) := y_0 \cdot z$, for all $z\in Z$. Therefore $x \cdot z = \lambda (z)$, for all $x$, $z\in X$ and thus,
$R (x, \,y) = (\lambda(y), \, \lambda (\vartheta (x)))$, for all $x$, $y\in X$. Such a solution is symmetric if and only if $\lambda$ is the constant map, i.e. there exists $a\in X$ such that $\lambda(x) = a$, for all $x\in X$ and hence $R = R_a$ is the constant solution $R_a \, (x, \, y) = (a, \, a)$, for all $x$, $y\in X$.
We also note that any two constant solutions $R_a$ and $R_b$ on $X$ are isomorphic.
\end{example}

Using \prref{detFS} and \thref{refqm} for describing and classifying all solutions on a given set $X$ comes down to: (1) describing and classifying the set of all Kimura semigroups that can be defined on $X$, and (2) for a fixed Kimura semigroup structure $\cdot$ describing the set of all binary operations $\ast : X \times X \to X$ that satisfy \equref{25} and \equref{26} or, equivalently, all quasi-endomorphisms of $(X, \, \cdot)$.  In the sequel, we will work with binary operations $\ast$ instead of quasi-endomorphisms. As basic examples and for future use, we describe all solutions on a set $X$ such that the Kimura semigroup $(X, \cdot)$ is of the form ${}_fX$ or $X_f$ as described in \exref{exkimura} (\ref{item:5}).

\begin{examples}\exlabel{lrzero}
  \;
  \begin{enumerate}[(1)]
  \item\label{item:6} Let $f = f^2 : X \to X$ such that $(X, \cdot) = {}_fX$, that is $x\cdot y = f(x)$, for all $x$, $y\in X$. Let $R$ be a solution on $X$ of the form $R(x, \, y) = (f(x), \, x\ast y)$. Taking into account that $f$ is an idempotent map, all compatibility conditions \equref{25} and \equref{26} are reduced to $x \ast y = f(y)$, for all $x$, $y\in X$ and thus we obtain $R (x, \, y) = (f(x), \, f(y))$, the map denoted by $R^f$ in \equref{exfsih2}. Observe that $R^f$ is neither left nor right non-degenerate.

  \item\label{item:7} Now let $h = h^2 : X \to X$ such that $(X, \cdot) = X_h$, that is $x\cdot y = h(y)$, for all $x$, $y\in X$. We will describe all solutions on $X $ of the form $R(x, \, y) = (h(y), \, x\ast y)$. Using the second compatibility condition of \equref{25} yields $x \ast y = h (x \ast z)$, for all $x$, $y$, $z\in X$. Thus, there exists a map $f: X \to X$ (fix $z_0 \in X$ and define $f(x) := h (x \ast z_0$), for all $x\in X$) such that $x\ast y = f(x)$, for all $x$, $y\in X$.  Now, we can see that \equref{25} and \equref{26} hold if and only if $h\circ f = f = f\circ h$. Thus, we have proved that $R = R_{(h, \, f)}$ as defined in \equref{exfsih}. Moreover, $R_{(h, \, f)}$ is left non-degenerate if and only if $h = {\rm Id}_X$ and $R_{(h, \, f)}$ is right non-degenerate if and only if $f: X \to X$ is bijective.
  \end{enumerate}
\end{examples}

Using \prref{detFS} we can classify all non-degenerate solutions on $X$.

\begin{theorem}\thlabel{th:nedegenerate}
  Let $X$ be a set and $R: X\times X \to X\times X$ a map. Then:
  \begin{enumerate}[(1)]
  \item $R$ is a left-nondegenerate solution on $X$ if and only if there exists a map $f: X \to X$ such that for any $x$, $y\in X$:
    \begin{equation} \eqlabel{nedegf}
      R (x, \, y) = R_f (x, \, y) = \bigl( y, \, f(x) \bigl).
    \end{equation}

  \item $R$ is a right-nondegenerate solution on $X$ if and only if there exists a bijective map $f: X \to X$ such that $R =  R_f$.

  \item\label{item:autrfautf} Two solutions $(X, \, R_f)$ and $(X, \, R_g)$ as above are isomorphic if and only if there exists $\sigma \in \Sigma_X$ such that $g = \sigma \circ f \circ \sigma^{-1}$. Furthermore, ${\rm Aut} \, (X, \, R_f) = \Aut(f)$, the function centralizer group.
  \end{enumerate}

   In particular, if $X$ is a finite set with $|X| = n$, then the number of isomorphism classes of all left (resp. right) non-degenerate solutions on $X$ is equal to $d(n)$, the Davis number of $n$ (resp. $p(n)$, the Euler partition number).
\end{theorem}

\begin{proof}
  \begin{enumerate}[(1)]
  \item Assume first that $R$ is a left-nondegenerate solution on $X$, i.e. for any $x\in X$ the map $l_x : X \to X$, $l_x (y) = x\cdot y$, for all $y\in X$ is bijective. In particular, $l_x$ is injective and hence $(X, \cdot)$ is a left cancellative Kimura semigroup. It follows from \equref{24} that $y\cdot z = z$, for all $y$, $z\in X$, that is $(X, \cdot)$ is the right zero semigroup. By applying (\ref{item:7}) of \exref{lrzero} to $h := {\rm Id}_X$ we obtain that there exists a map $f: X \to X$, such that $x\ast y = f(x)$ and hence $R (x, \, y) = (y, \, f(x))$, for all $x$, $y\in X$, as desired. Conversely, the map $R_f$ is obviously left-nondegenerate as $l_x = {\rm Id}_X$, for all $x\in X$.

  \item Assume now that $R$ is a right-nondegenerate solution on $X$, i.e. for any $y\in X$ the map $r_y : X \to X$, $r_y (x) = x \ast y$ for all $x\in X$, is bijective. In particular, $r_y$ is injective and thus using \equref{26} we obtain that $x \cdot y = y$, for all $x$, $y\in X$, that is $(X, \cdot)$ is the right zero semigroup. By applying (\ref{item:7}) of \exref{lrzero} to $h := {\rm Id}_X$ we obtain a map $f: X \to X$, such that $R = R_f$. Finally, $R_f$ is a right-nondegenerate map if and only $f$ is bijective since $r_y = f$, for all $y\in X$.

  \item Let $f$, $g: X\to X$ be two maps. Then $\sigma : (X, \, R_f) \to (X, \, R_f)$ is an isomorphism in $\mathcal{FSE}$ if and only if $\sigma : (X, \cdot) \to (X, \cdot')$ is an isomorphism of Kimura semigroups and $\sigma (x \ast y) = \sigma (x) \ast' \sigma (y)$, for all $x$, $y\in X$. Since the Kimura semigroups $(X, \cdot)$ and $(X, \cdot')$ are both right zero semigroups the first condition reduces to $\sigma$ being bijective. On the other hand, the compatibility $\sigma (x \ast y) = \sigma (x) \ast' \sigma (y)$ is equivalent to having $\sigma \circ f = g \circ \sigma$, that is, $f$ and $g$ are conjugate maps. The last statements now follows. \qedhere
  \end{enumerate}
\end{proof}

The next result describes and classifies a very large class of solutions on $X$ that includes surjective (in particular, bijective), of finite order or unitary solutions.

\begin{theorem}\thlabel{th:main}
  Let $R: X\times X \to X\times X$ be a solution on a set $X$ such that $\pi_1 \circ R : X\times X \to X$ is surjective. Then:

  \begin{enumerate}[(1)]
  \item\label{item:10} There exist two non-empty sets $A$ and $B$, a map $\omega: B\to B$ and an isomorphism of solutions $(X, \, R) \cong (A \times B, \, R_{\omega})$, where
    $R_{\omega}: (A\times B)^2 \to (A\times B)^2$ is the solution on $A\times B$ given for any $a_1$, $a_2 \in A$ and $b_1$, $b_2 \in B$ by:
    \begin{equation}\eqlabel{main1}
      R_{\omega} \bigl( (a_1, \, b_1), \, (a_2, \, b_2) \bigl) = \bigl( (a_1, \, b_2), \, (a_2, \, \omega(b_1) ) \bigl).
    \end{equation}

  \item\label{item:11} Let $(A \times B, \, R_{\omega})$ and $(C \times D, \, R_{\omega'})$ be two solutions of the form \equref{main1}. Then there exists a bijection between the set of all
    isomorphisms $(A \times B, \, R_{\omega}) \cong (C \times D, \, R_{\omega'})$ in the category $\mathcal{FSE}$ and the set of all pairs $(\sigma_1, \, \sigma_2)$
    consisting of two bijections $\sigma_1 : A \to C$ and $\sigma_2: B \to D$ such that $\omega' = \sigma_2 \circ \omega \circ \sigma_2 ^{-1}$.

  \item\label{item:12} There exists an isomorphism of groups:
    \begin{equation*}
    {\rm Aut} \, (X, \, R) \, \cong \, {\rm Aut} (A \times B, \, R_{\omega}) \, \cong \, \Sigma_A \times \Aut(\omega).
    \end{equation*}
  \end{enumerate}
\end{theorem}

\begin{proof}
  \begin{enumerate}[(1)]
  \item Let $R$ be a solution on $X$ such that $\pi_1 \circ R : X\times X \to X$ is surjective. Thus, $(X, \cdot)$ is a Kimura semigroup and since $\pi_1 \circ R : X\times X \to X$ is surjective, this semigroup is total, i.e. $X = X^2 = \{ x\cdot y \, | \, x, y \in X \}$. Using \thref{structura} we obtain that there exist two non-empty sets $A$ and $B$ and an isomorphism of semigroups $(X, \cdot) \cong A \times B$, where $A\times B$ is the rectangular band of the sets $A$ and $B$ as defined in \equref{recband}.  Without loss of generality we can assume that the Kimura semigroup $(X, \cdot)$ is equal to the rectangular band $A \times B$ and we will describe all solutions on it. Let $R: A\times B \times A \times B \to A\times B \times A \times B$ be a solution on $A\times B$. Since the multiplication $\cdot$ on the rectangular band $A \times B$ is given by \equref{recband} the map $R$ takes the form
    \begin{equation} \eqlabel{proviz}
      R \bigl( (a_1, \, b_1), \, (a_2, \, b_2) \bigl) = \bigl( (a_1, \, b_2), \, (a_1, \, b_1) \ast (a_2, \, b_2) \bigl)
    \end{equation}
    for all $a_1$, $a_2\in A$ and $b_1$, $b_2\in B$, where $\ast : A\times B \times A \times B \to A\times B$ is a binary operation on $A\times B$. Thus, we can find two maps $f: A\times B \times A \times B \to A$ and $g: A\times B \times A \times B \to B$ ($f = \pi_1 \circ \ast$, $ g = \pi_2 \circ \ast)$ such that for any $a_1$, $a_2 \in A$ and $b_1$, $b_2 \in B$ we have:
    \begin{equation*}
    (a_1, \, b_1) \ast (a_2, \, b_2) = \bigl( f(a_1, \, b_1, \, a_2, \, b_2), \, g(a_1, \, b_1, \, a_2, \, b_2) \bigl).
    \end{equation*}
    Next we have to check when such a map satisfies the compatibility conditions \equref{25} and \equref{26}. By considering $x = (a_1, \, b_1)$, $y = (a_2, \, b_2)$ and $z = (a_3, \, b_3)$ an immediate calculation shows that \equref{25} holds if and only if
    \begin{eqnarray*}
      && f(a_1, \, b_1, \, a_2, \, b_2) = a_2 \\
      && g(a_1, \, b_1, \, a_2, \, b_2) = g(a_1, \, b_1, \, a_3, \, b_3)
    \end{eqnarray*}
    for all $a_1$, $a_2$, $a_3 \in A$ and $b_1$, $b_2$, $b_3 \in B$. Hence, $g$ does not depend on the variables in positions $3$ and $4$. Thus, there exists a map $h: A\times B \to B$ ($h (a_1, \, b_1) := g (a_1, \, b_1, \, a_0, \, b_0)$, for some fixed elements $a_0 \in A$ and $b_0 \in B$) such that the binary operation $\ast$ is given by
    \begin{equation*}
    (a_1, \, b_1) \ast (a_2, \, b_2) = \bigl( a_2, \, h(a_1, \, b_1) \bigl)
    \end{equation*}
    for all $a_1$, $a_2\in A$ and $b_1$, $b_2\in B$. It remains to check the compatibility condition \equref{26} for such a map. Again, by taking $x = (a_1, \, b_1)$, $y = (a_2, \, b_2)$ and $z = (a_3, \, b_3)$ we can see that \equref{26} holds if and only if $h (a_1, \, b_2) = h (a_2, \, b_2)$, for all $a_1$, $a_2\in A$ and $b_2\in B$. Thus,
    $h$ does not depend on the variable in the first position, i.e. there exists a map $\omega: B\to B$ ($\omega (b) := h (a_0, \, b)$, for a fixed $a_0 \in A$) such that
    the binary operation $\ast$ is given by
    \begin{equation*}
    (a_1, \, b_1) \ast (a_2, \, b_2) = \bigl( a_2, \, \omega (b_1) \bigl)
    \end{equation*}
    for all $a_1$, $a_2\in A$ and $b_1$, $b_2\in B$. This leads us to the defining formula of $R_{\omega}$ as given in
    \equref{main1} which concludes the first part of the theorem.

  \item Let $(A \times B, \, R_{\omega})$ and $(C \times D, \, R_{\omega'})$  be two solutions of the form \equref{main1} for some maps $\omega: B\to B$ and
    $\omega' : D\to D$. Then $(A \times B, \, R_{\omega}) \cong (C \times D, \, R_{\omega'})$ in the category $\mathcal{FSE}$ if and only if there exists an
    isomorphism of Kimura semigroups $\psi : (A\times B, \cdot) \to (C\times D, \cdot')$ (i.e. of rectangular bands) such that
    \begin{equation} \eqlabel{clasprof}
      \psi (x \ast y) = \psi(x) \ast' \psi(y)
    \end{equation}
    for all $x$, $y\in A\times B$. Now, by applying \leref{morfrecta}, we obtain two unique bijective maps $\sigma_1: A\to C$, $\sigma_2: B\to D$ such that
    $\psi (a, \, b) = \bigl(\sigma_1 (a), \, \sigma_2(b) \bigl)$, for all $a \in A$ and $b\in B$. Furthermore, the correspondence $\psi \leftrightarrow (\sigma_1, \, \sigma_2)$
    is bijective. We are only left to check the compatibility condition \equref{clasprof}. By considering
    $x = (a, \, b)$ and $y = (a', \, b')$ we obtain easily that $\psi (x \ast y) = \bigl( \sigma_1 (a'), \, \sigma_2 (\omega(b)) \bigl)$ and
    $\psi(x) \ast' \psi(y) = \bigl( \sigma_1 (a'), \, \omega' (\sigma_2(b)) \bigl)$. Thus, \equref{clasprof} holds if and only if
    $\sigma_2 \circ \omega = \omega' \circ \sigma_2$, as needed.

  \item It follows directly from (2) applied for $(C \times D, \, R_{\omega'}) := (A \times B, \, R_{\omega})$. \qedhere
  \end{enumerate}
\end{proof}

\begin{remarks}\relabel{singel}
  \;
  \begin{enumerate}[(1)]

  \item Let $R: X\times X \to X\times X$ be a solution on a set $X$ such that $\pi_1 \circ R : X\times X \to X$ is surjective. If $B = \{\star\}$ is the singleton set in the decomposition $(X, \, R) \cong (A \times B, \, R_{\omega})$ given by \thref{th:main}, then $(X, \, R) \cong (X, \, {\rm Id}_{X\times X})$. On the other hand, if $A = \{\star\}$ is the singleton set, then $(X, \, R) \cong (X, \, R_f)$, for some map $f: X \to X$, where $R_f$ is the left non-degenerate solution given by \equref{nedegf}. Furthermore, if $X$ is a set with at least two elements, then it can be easily seen that  $(X, \, {\rm Id}_{X\times X})$ and $(X, \, R_f)$ are not isomorphic solutions in the category $\mathcal{FSE}$.

  \item\label{degenerate_solutions} \thref{th:nedegenerate} shows that if $R$ is a right non-degenerate solution on a set $X$, then $R$ is also left-nondegenerate and, moreover, $R$ is bijective. Now, for a set $A \neq \{\star\}$ and a permutation $\omega: B\to B$ on $B$, the map $R_{\omega}: (A\times B)^2 \to (A\times B)^2$ on $A\times B$ given by \equref{main1} is an example of a bijective solution which is left and right degenerate since $l_{(a_1, \, b_1)} \, (a_2, \, b_2) = (a_1, \, b_2)$,
    and $r_{(a_2, \, b_2)} \, (a_1, \, b_1) = (a_2, \, \omega (b_1))$, for all $a_1$, $a_2 \in A$ and $b_1$, $b_2 \in B$.
  \end{enumerate}
\end{remarks}

As a special case of \thref{th:main} (\ref{item:11}) we obtain that two solutions $(A \times B, \, R_{\omega})$ and $(A \times B, \, R_{\omega'})$ of the form \equref{main1} are isomorphic if and only if the corresponding maps $\omega$, $\omega' : B\to B$ are conjugate. Thus, for two fixed sets $A$ and $B$, the number of isomorphism classes of all solutions of the form \equref{main1} is $d(B)$, the David number of $B$.

For a positive integer $n$ let us denote by ${\rm fs}^1 \, (n)$ (resp. ${\rm fs}_{\rm b} \, (n)$) the number of isomorphism classes of all solutions $R$ on finite set with $n$ elements such that $\pi_1 \circ R $ is surjective (resp. $R$ is bijective). Applying \thref{th:main} for finite sets yields:

\begin{corollary} \colabel{cazfinit}
Let $n$ be a positive integer. Then
\begin{eqnarray}
{\rm fs}^1 \, (n) = \sum_{m|n} \, d(m), \qquad {\rm fs}_{\rm b} \, (n) = \sum_{m|n} \, p(m) \eqlabel{formdn0}
\end{eqnarray}
where $d(m)$ is the Davis number of $m$ and $p(m)$ is the number of all partitions of $m$.
\end{corollary}

By applying \thref{th:main} and \coref{cazfinit} we obtain:

\begin{corollary} \colabel{cazprim}
Let $q$ be a prime number. Then  ${\rm fs}^1 \, (q) = 1 + d(q)$ and ${\rm fs}_{\rm b} \, (q) = 1 + p(q)$. Furthermore, for a finite set $X$ with $|X| = q$,
any solution $R$ on $X$ with $\pi_1 \circ R : X\times X \to X$ surjective, is either ${\rm Id}_{X\times X}$ or  isomorphic to a solution of the form $R_f$ given by \equref{nedegf}, for some map $f : X \to X$.
\end{corollary}

Using \coref{cazfinit} and the M\H{o}bius inversion formula we obtain two new formulas for $d(n)$ and $p(n)$:

\begin{corollary} \colabel{nouaminune}
Let $n$ be a positive integer. Then:
\begin{eqnarray}
d(n) = \sum_{m|n} \, \mu(m) \, {\rm fs}^1 \, (\frac{n}{m}), \qquad  p(n) =  \sum_{m|n} \, \mu(m) \, {\rm fs}_{\rm b} \, (\frac{n}{m})  \eqlabel{formdn}
\end{eqnarray}
where $\mu : \NN^* \to \{-1, \, 0, \, 1\}$ is the M\H{o}bius function.
\end{corollary}

For unitary solutions on $X$, i.e. those $R : X \times X \to X\times X$ with $R^{21} R = {\rm Id}_{X\times X}$ (where $R^{21} = \tau_X R \tau_X$), we have:

\begin{corollary}\colabel{cor:unitary} Let $X$ be a set. Then:
  \begin{enumerate}[(1)]
  \item\label{item:8} The unitary solutions on $X$ are precisely those isomorphic to one of the form \equref{main1} for some decomposition $X\cong A\times B$ and involution $\omega:B\to B$.

  \item\label{item:9} If $X$ is a finite set with $|X| = n$, then the number of isomorphism classes of unitary solutions on $X$ is equal to
    \begin{equation*}
      \sum_{d|n}\left(1+\left\lfloor \frac d2\right\rfloor\right)
      =
      \tau(n) + \sum_{d|n}\left\lfloor \frac d2\right\rfloor.
    \end{equation*}
  \end{enumerate}
\end{corollary}

\begin{proof}
  We prove the two statements in the order in which they were made.
  \begin{enumerate}[(1)]
  \item A unitary solution will obviously be bijective, so \thref{th:main} applies; set $X = A\times B$ and let $\omega: B\to B$ be a map as in that theorem. A straightforward computation shows that
    \begin{equation*}
      (\tau R_{\omega} \tau ) R_{\omega} \bigl( (a_1, \, b_1), \, (a_2, \, b_2) \bigl) = \bigl( (a_1, \, \omega^2 (b_1) ), \, (a_2, \, b_2) \bigl)
    \end{equation*}
    for all $a_1$, $a_2 \in A$ and $b_1$, $b_2 \in B$. $R_{\omega}$, given by \equref{main1}, is thus unitary if and only if $\omega^2 = {\rm Id}_B$.

  \item By (\ref{item:8}), specifying a solution comes down to specifying a decomposition $X\cong A\times B$ and an involution on $B$. \thref{th:main} (\ref{item:11}) ensures that up to isomorphism, the decomposition in question amounts to specifying the size $d:=|B|$, which must be a divisor of $n=|X|$.

    As for the involution $\omega:B\to B$, it will be uniquely determined, up to isomorphism, by its number of fixed points: the other elements of $B$ then split into pairs
    \begin{equation*}
      \omega(x) = x'\ne x = \omega(x'),
    \end{equation*}
    and up to isomorphism, it does not matter how we effect this pairing. The number of isomorphism classes of involutions on $B$ is thus the number of non-negative integers $\le d:=|B|$ having the same parity as $d$, i.e. $1+\left\lfloor\frac d2\right\rfloor$. This gives the desired result. \qedhere
  \end{enumerate}
\end{proof}

\coref{invo} classifies all solutions of order $2$ in the group of permutations on $X\times X$. We can now take a step forward and describe all solutions of finite order $\geq 2$, for a given set $X$. To this end, we need to recall two numbers that are well studied in the combinatorics of permutation groups. For a positive integer $n$, the \emph{Landau number} $g(n)$ is defined as the maximum order of a permutation in $S_n$; it was introduced by Landau in \cite{land} where the first asymptotic formula was proved: $ {\rm log} \bigl( g(n) \bigl) \sim \sqrt{n \, {\rm log} (n)}$. We also recall that the number of elements of order $k$ in the symmetric group $S_n$ is denoted by $T (n, k)$, where $1 < k \leq g(n)$. These two numbers are well studied in the theory of permutation groups: for more details, including generating functions and recurrence relations, we refer to \cite{nicol, wilf} and their references.  Moreover, we denote by $\overline{T (n, k)}$, the number of all conjugation classes of permutations of order $k$ in the symmetric group $S_n$, for all $1 < k \leq g(n)$. Using \thref{th:main} and these numbers we classify all solutions of finite order:

\begin{corollary} \colabel{finitord}
Let $R$ be a solution of finite order $\geq 2$ on a set $X$ and $m$ a positive integer. Then the order ${\rm o}(R)$ of $R$ is even and ${\rm o}(R) = 2m$ if and only if  $(X, \, R) \cong (A \times B, \, R_{\omega})$, for a permutation $\omega: B \to B$ of order $m$ in $\Sigma_B$.

In particular, if $X$ is a finite set with $|X| = n$, then the maximum order of a solution on $X$ is $2\, g(n)$ and the number of isomorphism classes of
solutions of finite order $\geq 2$ on $X$ is equal to
\begin{equation*}
    \sum_{d|n} \, \Bigl(\sum_{k=1}^{g(d)} \, \overline{T (d, k)} \Bigl)
  \end{equation*}
where $g(d)$ is the Landau number of $d$.
\end{corollary}

\begin{proof}
Let $R$ be a solution of finite order $\geq 2$ on $X$. In particular, $X \neq \{\star\}$ and $R$ is bijective. Thus, we can apply \thref{th:main}: $X$ decomposes as $X \cong A \times B$ and let $R = R_{\omega}$ given by \equref{main1}, for a permutation $\omega \in \Sigma_B$. If $B = \{\star\}$, the singleton set, then $R_{\omega}$ is the identity map which leads to a contradiction. Thus, $B \neq  \{\star\}$ and for a positive integer $n$ we have:
\begin{eqnarray*}
&& R^{2n} \bigl( (a_1, \, b_1), \, (a_2, \, b_2) \bigl) = \bigl( (a_1, \, \omega^n (b_1) ), \, (a_2, \, \omega^n (b_2) ) \bigl) \\
&& R^{2n+1} \bigl( (a_1, \, b_1), \, (a_2, \, b_2) \bigl) = \bigl( (a_1, \, \omega^n (b_2) ), \, (a_2, \, \omega^{n+1} (b_1) ) \bigl)
\end{eqnarray*}
for all $a_1$, $a_2 \in A$ and $b_1$, $b_2 \in B$. Since $B \neq \{\star\}$ it follows that the order of $R_{\omega}$ cannot be odd and
${\rm o}(R_{\omega}) = 2n$ if and only if $\omega$ is a permutation of order $n$ in $\Sigma_B$. The proof is now finished by using (\ref{item:11}) of \thref{th:main}.
\end{proof}

The right hand version of \thref{th:main} is the following:

\begin{corollary}\colabel{pi2Rsurj}
Let $R: X\times X \to X\times X$ be a solution on a set $X$ such that $\pi_2 \circ R : X\times X \to X$ is surjective. Then there exist two non-empty sets $A$, $B$, a surjective map $\omega: B\to B$ and an isomorphism of solutions $(X, \, R) \cong (A \times B, \, R_{\omega})$, where $R_{\omega}: (A\times B)^2 \to (A\times B)^2$ is the solution on $A\times B$ given by \equref{main1}.

If $X$ is a finite set with $|X| = n$, then the number of isomorphism classes of all solutions on $X$ such that $\pi_2 \circ R : X\times X \to X$ is surjective is
equal to $\sum_{m|n} \, p(m)$.
\end{corollary}

\begin{proof}
From \thref{refqm} we obtain that any solution $R$ on $X$ is of the form $R = R_{(\cdot, \, \vartheta)}$ as given by \equref{toatesol},
for a quasi-endomorphism $\vartheta : X \to X$ of the Kimura semigroup $(X, \cdot)$. Now, since $\pi_2 \circ R_{(\cdot, \, \vartheta)}$ is surjective,
we obtain that the Kimura semigroup $(X, \cdot)$ is total and by \thref{structura} there exist two non-empty sets $A$ and $B$ and an isomorphism of semigroups $(X, \cdot) \cong A \times B$, where $A\times B$ is the rectangular band of the sets $A$ and $B$ as defined by \equref{recband}.
Now, the proof of \thref{th:main} shows in fact that any solution of the rectangular band $A\times B$ has the form $R_{\omega}$ given by \equref{main1}, for a map $\omega: B\to B$. Of course,
$\pi_2 \circ R_{\omega}$ is surjective if and only if $\omega: B\to B$ is surjective.
\end{proof}

\section{The classification of indecomposable solutions}\selabel{sect3}

Following \cite[Definition 2.5]{etingof} we recall the following notion:

\begin{definition}\delabel{def:dec}
  A solution $R: X\times X \to X\times X$ on a set $X$ is called
  \begin{itemize}
  \item {\it decomposable} if there is a non-trivial partition $X=Y\sqcup Z$ with
    \begin{equation*}
      R(Y\times Y)\subseteq Y\times Y
      \quad\text{and}\quad
      R(Z\times Z)\subseteq Z\times Z;
    \end{equation*}
  \item {\it indecomposable} if it is not decomposable in the sense above.
  \end{itemize}
\end{definition}

When dealing with (in)decomposable solutions we assume $X \neq \{\star\}$, the singleton set.

\begin{remarks}\relabel{clasideqybe}
  \;
  \begin{enumerate}

  \item\label{item:gripe} The term `indecomposable solution' of \deref{def:dec} is by now firmly entrenched in the Yang-Baxter literature going back some decades, so we have kept it. We do, however, have some reservations on its suitability. One issue is that the phrase suggests an analogy to decomposability in abelian categories (e.g. modules), but such an analogy would be misleading: while a partition $X=Y\sqcup Z$ as in \deref{def:dec} will certainly give back solutions $R_Y$ and $R_Z$ on $Y$ and $Z$ respectively, this procedure does not seem to be canonically reversible. Indeed, starting with solutions $R_Y$ and $R_Z$ on $Y$ and $Z$, it is unclear how a solution on $Y\sqcup Z$ built out of the two should operate on elements $(y,z)\in Y\times Z$.

    It will also become apparent below (\thref{th:conngrph}) that this notion of indecomposability is intimately linked to graph connectedness. For that reason, `connected' might have been a more aptly descriptive term. This, indeed, is what \cite[paragraph immediately preceding \S 2, p.323]{zbMATH03277468} suggests as well. 

    There are also, arguably, good alternative uses for the term `indecomposable' itself: it might refer to those solutions not expressible as a (non-obvious) product in the sense of \exref{exemFS} (\ref{item:prodsol}). Yet another alternative, motivated by \thref{refqm}, would be to have it refer to solutions whose associated pointed Kimura semigroups are indecomposable in the sense of semigroup theory: $(X, \, \vartheta)$ cannot be written as a Zappa-Sz\'{e}p product of two sub-objects. We believe these two types of solutions might be interesting to study.

  \item For involutive non-degenerate solutions of the Yang-Baxter equation, $(X, \, R)$ is indecomposable precisely when the canonical group
  $\mathcal{G} \, (X,\, R)$ associated to $R$ acts transitively on $X$ \cite[Proposition 2.11]{etingof}. The classification of all \emph{indecomposable, nondegenerate and involutive} set-theoretic solutions of the Yang-Baxter equation on a set $X$ with $|X| = p$ (resp. $|X| = p q$), where $p$ and $q$ are distinct prime numbers, was given in \cite[Theorem 2.12]{etingof} (resp. \cite{APPJ23}). While on a set with $p$ elements there is, up to isomorphism, only one solution, on a set with $pq$ elements the number of isomorphism classes of all involutive non-degenerate indecomposable solutions is $1 + \frac{p^q - p}{p-1} + \frac{q^q - q}{q-1}$.

  Contrary to the Yang-Baxter equation case, the variety of indecomposable solutions of the Frobenius-Separability equation is narrower as we will see in \coref{camputine}. For instance, any involutive solution is decomposable and, on a set with $n$ elements there is, up to isomorphism, only one indecomposable and right non-degenerate solution, namely the one associated to a cycle of length $n$ (\coref{clasinderighnede}).
\end{enumerate}
\end{remarks}

\begin{examples}\exlabel{exdeindecom}
  \;
  \begin{enumerate}[(1)]
  \item  For any set $X$, the identity map ${\rm Id}_{X\times X}$ and the flip map $\tau_{X} : X\times X \to X \times X$, $\tau_{X} \, (x, \, y) = (y, \, x)$
  are decomposable  solutions on $X$.

  \item If $a\in X$ is a fixed element of a set $X$, then the constant solution $R_a (x, \, y) := (a, \, a)$, for all $x$, $y\in X$ is an indecomposable solution on $X$. \coref{idemindecom} shows that the constant solutions are the only indecomposable idempotent solutions on $X$.

   \item Let $R \colon X\times X \to X\times X$ be a solution on a set $X$. Applying \thref{refqm} we obtain a quasi-endomorphism $\vartheta : X \to X$ of the Kimura semigroup $(X, \cdot)$ such that $R = R_{(\cdot, \, \vartheta)}$ as given by \equref{toatesol}. Thus, $R_{(\cdot, \, \vartheta)}$ is a decomposable solution if and only if there exists a non-trivial partition $X = Y\sqcup Z$, where $Y$ and $Z$ are subsemigroups of $(X, \cdot)$ such that $y \cdot \vartheta(Y) \subseteq Y$, $z\cdot \vartheta(Z) \subseteq Z$, for all $y\in Y$ and $z\in Z$.

   \item Let $f: X \to X$ be an injective map having a fixed point $a\in X$. Then $R_f :X \times X \to X\times X$ given by $R_f (x, \, y) = (y, \, f(x))$,
   for all $x$, $y\in X$ is a left-nondegenerate decomposable solution on $X$. Indeed, $Y := \{a\}$ and $Z := X \setminus \{a\}$ is a partition on $X$ and $R_f (a, \, a) = (a, \, a)$ and $R_{f} (Z\times Z)\subseteq Z\times Z$, since $f$ is injective.

  \item If $X$ is a finite set with $|X| = n$ and $f: X \to X$ is a cycle of length $n$, then $R_f :X \times X \to X\times X$ given by $R_f (x, \, y) = (y, \, f(x))$, for all $x$, $y\in X$ is a right-nondegenerate indecomposable solution on $X$.

\end{enumerate}
\end{examples}

As we mentioned above, the variety of indecomposable solutions is rather narrow. Indeed, we have:

\begin{corollary} \colabel{camputine}
Any solution $R: X\times X \to X\times X$ on a set $X$ such that $\pi_1 \circ R : X\times X \to X$ or $\pi_2 \circ R : X\times X \to X$ is surjective
is decomposable.

In particular, any surjective (e.g. bijective, involutive, unitary or of finite order) solution on $X$ is decomposable.
\end{corollary}

\begin{proof}
Applying \thref{th:main} or \coref{pi2Rsurj} we obtain that $X$ decomposes as $X \cong A \times B$ and $R = R_{\omega}$ is given by
\equref{main1}, for a map $\omega : B \to B$. The proof will be finished once we observe that $R_{\omega}$ is a decomposable solution. Indeed, fix an element
$a_0 \in A$ and define $Y: = \{ (a_0, \, b) \, | \, b\in B \}$ and $Z := A\times B \setminus Y$. Then $Y$ and $Z$ form a non-trivial partition of $A \times B$ and
taking into account the formula of $R_{\omega}$ given in \equref{main1} we have that $R_{\omega} (Y\times Y)\subseteq Y\times Y$ and
$R_{\omega} (Z\times Z)\subseteq Z\times Z$.
\end{proof}

Among the classes of solutions classified in this paper, we are left to determine which of the left/right non-degenerate (classified in \thref{th:nedegenerate}) and idempotent (classified in \coref{special_sol}) solutions are indecomposable. Both types of solutions are intimately linked to self-functions $f:X\to X$, decomposability transports over to functions. To study their indecomposability we recall, at this point, that functions can be conveniently cast as graphs as in \deref{def:grph}.

We will use some fairly standard graph-theoretic terminology, available, say, in \cite{Die00, harary_graph-book}. Whatever graphs we encounter will typically be directed (digraphs, for short), and in fact functional in the sense of \reref{re:fngrphs}. For that reason, one typically needs to distinguish among various levels of connectivity (see \cite[discussion preceding Theorem 16.1]{harary_graph-book}): a digraph is
\begin{itemize}
\item {\it connected} if its underlying {\it un}oriented graph is connected, in the usual sense (e.g. \cite[\S 1.4]{Die00}): any two vertices can be joined by a path, disregarding the edge orientations. This is the {\it weak connectedness} of \cite[discussion preceding Theorem 16.1]{harary_graph-book}.
\item {\it unilaterally connected} if for any two vertices $x,y$ there is an {\it oriented} path $x\to y$ or $y\to x$.
\item and finally, {\it strongly connected} if for any two vertices $x$ and $y$ there are oriented paths $x\to y$ and $y\to x$.
\end{itemize}

Our next result describes all left-nondegenerate indecomposable solutions on a set $X$.

\begin{theorem}\thlabel{th:conngrph}
Let $R = R_f$ be a left-nondegenerate solution on a set $X$ associated to a self-function $f: X \to X$ as in \thref{th:nedegenerate}. The following statements are equivalent:

\begin{enumerate}[(1)]
\item \label{item:g13} $R_f$ is indecomposable;

\item\label{item:14} The functional digraph $\Gamma_f$ is connected;

\item\label{item:15} For every $x$, $y\in X$ there are non-negative integers $m$ and $n$ with $f^m (x) = f^n (y)$;

\item \label{item:g15} There is no non-trivial $f$-invariant partition of $X$, i.e. $X = Y\sqcup Z$ such that $f(Y) \sqsubseteq Y$ and
$f(Z) \sqsubseteq Z$.\footnote{If $X$ is a finite set, then each of the equivalent conditions is also equivalent to the fact that the
functional digraph $\Gamma_f$ has exactly one cycle (\cite[Corrolary, pag. 2]{harary}).}
\end{enumerate}
A function $f: X \to X$ satisfying one of the above conditions will be called \emph{connected}.
\end{theorem}

\begin{proof}
  The set of vertices of every connected component $C$ of $\Gamma_f$ is $f$-invariant. This of course applies also to the {\it complement} of a connected component (since that complement is itself a union of connected components), so the connected components are the smallest $f$-invariant subsets of $X$ whose complements are also $f$-invariant. This proves that (\ref{item:g13}) is equivalent to (\ref{item:14}).

  As for (\ref{item:15}), note that it too defines an equivalence relation on $X$, with $f$ leaving every class $C$ invariant, and with all functional graphs $\Gamma_{f|_C}$ connected. This implies that said equivalence is nothing but that of being in the same connected component of $\Gamma_f$, hence the equivalence (\ref{item:14}) $\iff$ (\ref{item:15}). Finally, the equivalence (\ref{item:g13}) $\iff$ (\ref{item:g15}) follows from \deref{def:dec} taking into account that $R_f(x, \, y) = (y, \, f(x))$,
  for all $x$, $y\in X$.
\end{proof}

In order to indicate precisely the number of isomorphism classes of all left-nondegenerate indecomposable solutions we need the following:

\begin{definition}\delabel{hbnumber} For a positive integer $n$ we denote by $\mathfrak{c} (n)$ the number of all conjugation classes of all connected self-maps $f: \{1, \cdots, n\} \to \{1, \cdots, n\}$ and we call it the \emph{Harary number} of $n$.\footnote{Equivalently, $\mathfrak{c} (n)$ is the number of isomorphism classes of connected digraphs of the form $\Gamma_f$, for some $f: \{1, \cdots, n\} \to \{1, \cdots, n\}$.}
\end{definition}

\begin{remarks}\relabel{re:hararybruijn}
  \;
  \begin{enumerate}[(1)]

  \item Although apparently uncommon, the phrase `connected function' is also used in the same fashion in the description of the sequence \cite{slo-A002861}, whose first ten values are 1, 2, 4, 9, 20, 51, 125, 329, 862, 2311.

  \item The number $\mathfrak{c} (n)$ of \deref{hbnumber} appears obliquely in \cite[\S 4]{harary}: in that paper's notation, $\mathfrak{c}(n)$ is {\it almost} the coefficient of $x^n$ in $v_1(x)+v_2(x)+v_3(x)$; the qualification is necessary because, as explained in \reref{re:fngrphs}, \cite{harary} does not permit functions with fixed points.

    If $\Gamma_f$ is connected then $f$ can have at most {\it one} fixed point, in which case $\Gamma_f$ is a {\it directed rooted tree} (\cite[p.204]{harary} and \deref{def:root}). It follows, then, that having defined the power series $v_i$ as on \cite[p.209]{harary}, we have
    \begin{equation*}
      \sum_{n\ge 1}\mathfrak{c}(n)x^n = v_1(x) + v_2(x) + v_3(x) + T(x)
    \end{equation*}
    where $T(x)$ is the generating function for the number of rooted trees as in \reref{davids}.

  \item   The same $\mathfrak{c}(n)$ also makes an appearance as the $p_n$ on \cite[p.18]{bruijn}, and the generating function on \cite[p.19]{bruijn} relating $\mathfrak{c}(n)$ back to the Davis numbers $d(n)$ of \seref{prel} reads
    \begin{equation*}
      \sum_{n=0}^{\infty} \, d(n) \, x^n \, = \, \prod_{j=1}^{\infty} \, (1 - x^j)^{- \mathfrak{c} (j) }.
    \end{equation*}

  \end{enumerate}
\end{remarks}

Using the above combinatorial number we obtain:

\begin{corollary} \colabel{clasinderighnede}
Let $X$ be a finite set with $|X| = n$. Then the number of isomorphism classes of indecomposable left-nondegenerate solutions on $X$ is equal to the
Harary number $\mathfrak{c} (n)$.

Furthermore, up to isomorphism, there is only one right-nondegenerate indecomposable solution on $X$, the solution $R_f$ associated to a cycle $f: X \to X$ of length $n$.
\end{corollary}

\begin{proof}
The first statement follows from \thref{th:nedegenerate} and \thref{th:conngrph}. For the second statement, first using (2) of \thref{th:nedegenerate} we obtain that any right-nondegenerate solution of $X$ is given by $R_f$, for some permutation $f$ on $X$. Such a permutation
satisfies the third condition of \thref{th:conngrph} if and only if $f$ is a cycle of length $n$; applying now (3) of \thref{th:nedegenerate} we obtain that
any two solutions associated to two cycles of length $n$ are isomorphic.
\end{proof}

Finally, we classify all idempotent indecomposable solutions on a set $X$.

\begin{corollary} \colabel{idemindecom}
Up to isomorphism, there is only one idempotent indecomposable solution on any set $X$, namely the constant solution $R_a$ as defined
in \exref{exemFS} (\ref{item:constsol}), i.e. $R_a (x, \, y) = (a, \, a)$, for all $x$, $y\in X$.
\end{corollary}

\begin{proof}
First observe that the counterpart of \thref{th:conngrph} works for idempotent solutions on $X$ as well: indeed, using \coref{special_sol} any idempotent solution $R$ on $X$ is given by $R (x, \, y) = R^f(x, \, y) = (f(x), \, f(y))$, for an idempotent map $f =f^2 : X \to X$. By \deref{def:dec} it follows that
$R^f$ is an indecomposable solution on $X$ if and only if there is no non-trivial $f$-invariant partition of $X$ or, in other words, if $f$ is an idempotent connected
self-function on $X$. The proof will be finished once we show that the only functions with this property are the constant ones. To this end, let $f =f^2 : X \to X$ be an idempotent connected
map. If $f = {\rm Id}_X$, then using (\ref{item:15}) of \thref{th:conngrph}, we obtain that $X = \{\star\}$, the singleton set and we are done. Now let $a \neq b \in X$ such that $f(b) = a$. As $f$ is assumed to be connected and idempotent, we will prove that $f(x) = a$, for all $x \in X$. Indeed, as $f$ is idempotent we must have $f(a) = a$. Since $f$ is also connected, if $x\in X \setminus \{a, \, b\}$ it follows from (\ref{item:15}) of \thref{th:conngrph} that we can find two non-negative integers $m$ and $n$ such that $f^m (x) = f^n (a) = a$. As $f$ is idempotent, $f^m (x)$ must be either $x$ (impossible, since $x \neq a$) or $f(x)$. This shows that $f (x) = a$, for all $x \in X$ and consequently $R^f = R_a$, the constant solution. The proof is now finished since any two constant solutions $R_a$ and $R_b$ are isomorphic.
\end{proof}

\section{Function centralizers, connected maps and directed rooted trees}\selabel{sect4}

\thref{th:nedegenerate} suggests we have a closer look at automorphism groups of functions $f:X\to X$ (recall \deref{def:grupuriciud}). Note first the following immediate consequence of \thref{th:nedegenerate}:

\begin{corollary}\colabel{autfs}
Any function centralizer is isomorphic to the automorphism group of a left non-degenerate Frobenius-Separability (in particular, quantum Yang-Baxter) solution.  \qedhere
\end{corollary}

In view of this, the natural question of how broad a class the automorphism groups $\Aut(X,R)$ constitute (and which we will not answer fully) has a simpler offshoot: are all groups function centralizers? We address this here:

\begin{theorem}\thlabel{th:isnotcentr}
  \;
  \begin{enumerate}[(1)]
  \item\label{item:2fnok} For any cardinal number $\alpha\ge 2$, every group is an $\alpha$-function centralizer
  \item\label{item:1fnnotok} Finite groups admitting no surjections onto non-trivial cyclic groups are not function centralizers. In particular, not every group is a function centralizer.
  \end{enumerate}
\end{theorem}
\begin{proof}
  Part (\ref{item:2fnok}) is an immediate application of \cite[\S II.5, Theorem 5.5]{pt_bk}, which even shows that for every $\alpha$ every {\it monoid} $M$ can be realized as the endomorphism monoid $\End(f_a,\ a\in \alpha)$ of a set $X$ equipped with $\alpha$ functions $f_a$. The rest of the proof is concerned with (\ref{item:1fnnotok}).

  Let $f:X\to X$ be a set endomorphism. Its automorphism group $\Aut(f)$ is nothing but the automorphism group of the directed graph $\Gamma:=\Gamma_f$ of \deref{def:grph}, so we work with it in those terms.

  First, note the decomposition
  \begin{equation*}
    \Aut(\Gamma)\cong \prod_{i}\Aut(\Gamma_i),
  \end{equation*}
  where each $\Gamma_i\subset \Gamma$ is a union of connected components of $\Gamma$ with all connected components in $\Gamma_i$ mutually isomorphic and no components in $\Gamma_i$ isomorphic to any components in $\Gamma_j$, $j\ne i$. If a group $G$ does not admit a non-trivial direct-product decomposition, a realization $G\cong \Aut(\Gamma)$ thus implies that all components of $\Gamma$ are mutually isomorphic.

  Observe, next, that $\Aut(\Gamma)$ surjects onto the permutation group of the set of components of $\Gamma$. If $G$ does not admit a surjection onto a full permutation group, then $\Gamma$ must be connected.

  We have thus reduced the problem to realizing
  \begin{itemize}
  \item non-trivial groups $G$ with no non-trivial direct-product decomposition and no surjection onto any non-trivial full symmetric group;
  \item as automorphism groups $\Aut(\Gamma)$ for a function $f:X\to X$ with connected associated graph $\Gamma=\Gamma_f$.
  \end{itemize}
  If there is some vertex $x\in \Gamma$ fixed by all of $\Aut(\Gamma)$, then all vertices in the forward orbit $\Oo_{+,x}:=\{x,\ fx,\ f^2x,\ \cdots\}$ are also fixed, so a {\it non-trivial} automorphism group $G\cong \Aut(\Gamma)$ must fail to fix some vertex in the {\it backward} orbit
  \begin{equation*}
    \Oo_{-,x}:=\{x\}\cup f^{-1}x\cup f^{-2}x\cup\cdots.
  \end{equation*}
  We may as well assume that $f^{-1}x$ is not pointwise invariant under $G$, which means that $G$ will surject onto the full non-trivial permutation group of a set of components of $\Gamma\setminus \Oo_{+,x}$. Since we are assuming there is no such surjection, it must be the case that $G$ fixes {\it no} vertices $x\in \Gamma$.

  There are now a couple of cases to consider (cf. the classification of functional digraphs in \cite[Theorems 1 and 2 and Corollaries 1 and 2]{harary}).

  \begin{enumerate}[(I)]
  \item {\bf There is a finite forward orbit.} Suppose that orbit is of size $n$, consisting of
    \begin{equation*}
      x,\ fx,\ \cdots,\ f^{n-1}x,\ f^nx=x.
    \end{equation*}
    If $\Aut(f)$ fixes each of these elements, then in particular it has points fixed globally, which we are assuming is not the case. Otherwise, $\Aut(f)$ surjects onto a cyclic group permuting the $f^i x$ cyclically.

  \item {\bf There are no finite forward orbits.} It follows that $\Gamma=\Gamma_f$ (or rather its underlying undirected graph) is a {\it tree}. Since it acts with no fixed points, its automorphism group $\Aut(\Gamma)$ cannot have the {\it property (FA)} of \cite[\S I.6.1]{ser_tr}.

  \end{enumerate}

  There are many groups satisfying the flagged conditions:
  \begin{itemize}
  \item non-trivial, with no non-trivial direct-product decomposition;
  \item and no surjection onto a non-trivial full permutation group;
  \item or a non-trivial cyclic group;
  \item and with property (FA).
  \end{itemize}

  {\it Finite} groups with no surjections onto a non-trivial cyclic group will do (since they also fail to surject onto any non-trivial $S_n$).
\end{proof}

\begin{remark}\relabel{re:babai1}
  See \cite[Theorem 1]{babai_planar-1} for an analogue: not all finite groups are realizable as automorphism groups of finite planar graphs.
\end{remark}

Here, the motivation for considering groups $\Aut(f)$ to begin with was that they are also automorphism groups of $\fse$ solutions: by \thref{th:nedegenerate} (\ref{item:autrfautf}), $\Aut(X,R_f)=\Aut(f)$. \exref{exemFS} (\ref{item:fh}) discusses a somewhat broader class of examples: those of the form $R_{(h,f)}$, for
\begin{equation}\eqlabel{eq:fh}
  f, \, h:X\to X, \quad h^2=h,\quad fh=hf=f.
\end{equation}

Furthermore, that same discussion also observes that $\Aut\left(X,R_{(h,f)}\right)=\Aut(h,f)$. It is perhaps worth noting that this does not enlarge the class of automorphism groups.

\begin{proposition}\prlabel{pr:fhsinglefn}
  For functions $f,h:X\to X$ satisfying \equref{eq:fh}, the automorphism group is a single-function centralizer.
\end{proposition}
\begin{proof}
  Recasting the pair $(h,f)$ as an $\fse$ solution $R_{(h,f)}$ per \exref{exemFS} (\ref{item:fh}), the description of the automorphism group $\Aut(X,R_{(h,f)})$ in \thref{th:refqmbis} (\ref{item:automeansthis}) specializes (via \reref{re:exrevisited}) to
  \begin{equation*}
    \{\sigma\in \Aut(h)\ |\ \sigma\circ f|_{\im~h} = f\circ\sigma|_{\im~h}\}.
  \end{equation*}

  This is nothing but the centralizer of the function
  \begin{equation*}
    g(x):=
    \begin{cases}
      f(x)&\textrm{ if }x\in Y\\
      h(x)&\textrm{ if }x\in Z,
    \end{cases}
  \end{equation*}
  hence the conclusion.
\end{proof}

\begin{remark}\relabel{re:pb2}
  \coref{autfs}, \thref{th:isnotcentr} and \prref{pr:fhsinglefn} make it natural to wonder how broad the class of groups of the form $\Aut(X,R)$ is, for $\fse$ solutions $(X,R)$. The same could be asked, of course, of the set-theoretic Yang-Baxter equation.
\end{remark}

In analyzing groups of the form $\Aut(f)$ for single functions $f:X\to X$, it will be useful to observe that they decompose according to the partition of $X$ into connected components of $\Gamma_f$. Specifically:

\begin{proposition}\prlabel{pr:autgpdec}
  For a function $f:X\to X$ we have a decomposition
  \begin{equation}\eqlabel{eq:arbfconnf}
    \Aut(f)\cong \prod_i \left(\Aut(f_i)\wr \Sigma_{C_i}\right),
  \end{equation}
  where
  \begin{itemize}
  \item $X=\coprod_{s\in S} X_s$ is the disjoint union into connected components of $\Gamma:=\Gamma_f$;
  \item the partition
    \begin{equation*}
      S=\coprod_i C_i
    \end{equation*}
    is the coarsest so that for each $i$ all $\Gamma_{f|_{X_s}}$, $s\in C_i$ are mutually isomorphic;
  \item $f_i$ denotes any one of the restrictions $f|_{X_s}$, $s\in C_i$ (it does not matter which, since all are isomorphic);
  \item and `$\wr$' denotes the {\it wreath product} \cite[Chapter 7, preceding Theorem7.24]{rot_gp}.
  \end{itemize}
\end{proposition}
\begin{proof}
  We observed before (in the proof of \thref{th:isnotcentr}) that we have $\Aut(f)\cong \Aut(\Gamma)$. An automorphism permutes the connected components $\Gamma_s$ of $\Gamma$ with respective vertex sets $X_s$, operating, by assumption, independently on the distinct $C_i$ (since no $\Gamma_s$, $s\in C_i$ is isomorphic to any $\Gamma_s'$, $s'\in C_j$ if $i\ne j$).

  This much already gives a product decomposition
  \begin{equation*}
    \Aut(f)\cong \prod_i \Aut(f|_{X_i}),
  \end{equation*}
  where
  \begin{equation*}
    X_i:=\coprod_{s\in C_i}X_s.
  \end{equation*}
  Now note that all connected components of $\Gamma_{f|_{X_i}}$ are isomorphic to $\Gamma_{f_i}$ (in our statement's notation). Having made all of these identifications, the decomposition
  \begin{equation*}
    \Aut(f|_{X_i}) \cong \Aut(f_i)\wr \Sigma_{\textrm{set of components of }\Gamma_{f|_{X_i}}} = \Aut(f_i)\wr \Sigma_{C_i}
  \end{equation*}
  follows.
\end{proof}

\begin{remark}\relabel{re:prins}
  Cf. the decomposition of the automorphism group of a finite tree as an iteration of products and wreath products in \cite[Theorem 4.2.1]{prins_phd}. It is not uncommon in older literature to refer to wreath products as {\it Kranz products} \cite[\S 4]{babai_planar-2}; in \cite[\S IV.1]{prins_phd} the term is {\it `Kranzgroup'}.
\end{remark}

In view of \prref{pr:autgpdec}, it makes some sense to focus our attention on automorphism groups of {\it connected} functional graphs. As shorthand for this phrase, we will sometimes refer to $f:X\to X$ itself as \emph{connected} if its underlying functional digraph $\Gamma_f$ is (\thref{th:conngrph}). Additionally, following \cite[\S 2]{harary}:

\begin{definition}\delabel{def:root}
  A {\it directed rooted tree} is a directed tree with a distinguished vertex, the {\it root}, so that all edges are oriented towards the root.

  In particular, the root is a {\it sink}: it is adjacent only to incoming edges, and no outgoing ones.
\end{definition}

\begin{remark}\label{re:isrooted}
  A functional digraph $\Gamma_f$ is a rooted tree precisely when $f$ is connected and has a unique fixed point (which will then be the root).

  We leave this to the reader as a simple exercise; although not stated in precisely this fashion, it also follows from the discussion surrounding \cite[Theorems 1 and 2]{harary}.
\end{remark}

The following result catalogs the various types of behavior that can be expected from $\Aut(f)$ for connected $f$.

\begin{theorem}\thlabel{th:autogp-types}
  Let $f:X\to X$ be a connected function. Exactly one of the following situations obtains:
  \begin{enumerate}[(a)]
  \item\label{item:cycle} $f$ has a finite orbit. In this case it has exactly one cyclic orbit, and $\Aut(f)$ is a wreath product of the form
    \begin{equation}\eqlabel{eq:wrzs}
      \Aut(f)\cong \left(\prod_{i=1}^t G_i\right)\wr \ZZ/s
    \end{equation}
    where $G_i$ are automorphism groups of directed rooted trees, $st$ is the size of the unique cyclic orbit of $f$, and the wreath product is with respect to a cyclic action of $\ZZ/s$ on a size-$s$ set.

  \item\label{item:nofix} There are no finite orbits, but there are elements of $\Aut(f)$ fixing no points of $X$. In this case we have
    \begin{equation}\eqlabel{eq:wrz}
      \Aut(f)\cong \left(\prod_{i=1}^t G_i\right)\wr \ZZ,
    \end{equation}
    where $G_i$ are automorphism groups of directed rooted trees and the wreath product is with respect to the translation action of $\ZZ$ on itself.

  \item\label{item:everyonefix} There are no finite orbits and every element of $\Aut(f)$ fixes some point. In this case $\Aut(f)$ is a countable union of countable products of directed-rooted-tree automorphism groups.

  \end{enumerate}
\end{theorem}
\begin{proof}
  The graph $\Gamma_f$ either does or does not contain an oriented cycle, and if it does that cycle is unique (because once entered it cannot be escaped, contradicting connectedness and hence condition (\ref{item:15}) of \thref{th:conngrph}). Moreover, a cycle exists precisely when there are finite orbits, for any finite orbit will eventually enter a cycle. This gives the dichotomy between (\ref{item:cycle}) on the one hand and (\ref{item:nofix}) and (\ref{item:everyonefix}) on the other. We discuss the two cases separately (as in the proof of \thref{th:isnotcentr}).

  \begin{enumerate}[]
  \item {\bf Case (\ref{item:cycle}): $\Gamma_f$ has a cycle.} Suppose that cycle has length $n$:
    \begin{equation*}
      x_0\xrightarrow{\ f\ }x_1\xrightarrow{\ f\ }\cdots\xrightarrow{\ f\ } x_{n-1}\xrightarrow{\ f\ } x_0
    \end{equation*}
    (this includes the case $n=1$, meaning $f$ has a fixed point). The automorphism group $\Aut(f)$ operates on that cycle preserving the orientation, hence as a cyclic group of some order $s|n$. Setting $t:=\frac ns$, the $G_i$ of \equref{eq:wrzs} are the automorphism groups of the functional rooted trees associated to the restrictions of $f$ to the backward orbits
    \begin{equation*}
      \Oo_{-,x_i} = \{x_i\}\cup f^{-1}x_i\cup\cdots\textrm{ for } 0\le i\le t-1.
    \end{equation*}

  \item {\bf Cases (\ref{item:nofix}) and (\ref{item:everyonefix}): $\Gamma_f$ has no cycles.} It must then be a functional tree with no sinks (for a sink would be a fixed point, for $f$, which we are assuming does not exist).

    Regarding $\ZZ$ as the vertex set of a bi-infinite oriented path, there is a digraph morphism $|\cdot|:\Gamma_f\to \ZZ$ defined as follows:
    \begin{itemize}
    \item set $|x_0|=0$ for a fixed $x_0\in X$;
    \item for any $x\in X$, having found $m,n\in \ZZ_{\ge 0}$ with $f^m x_0=f^n x$ (which we always can, $f$ being connected: \thref{th:conngrph}), set
      \begin{equation*}
        |x|:=m-n.
      \end{equation*}
    \end{itemize}
    We leave it to the reader to check that this map is well defined, and further allows us to define a group morphism
    \begin{equation}\eqlabel{eq:amp}
      \Aut(f)\ni \alpha\xmapsto{\quad\amp\quad}|\alpha(x)|-|x|\in (\ZZ,+),
    \end{equation}
    independent of $x\in X$ (`$\amp$' for `amplitude' \cite[\S 2.2]{tits_arbre}). An automorphism $\alpha\in \Aut(f)$ fixes some point precisely when $\amp(\alpha)=0$, so the distinction between cases (\ref{item:everyonefix}) and (\ref{item:nofix}) is precisely that between $\amp$ being trivial and, respectively, not.

    Suppose first that we are in case (\ref{item:nofix}), i.e. \equref{eq:amp} is {\it not} trivial. Its image is then some $t\ZZ$, $t\in \ZZ_{>0}$, so choose some
    \begin{equation*}
      \alpha\in \Aut(f)\quad\textrm{with}\quad \amp(\alpha)=t.
    \end{equation*}
    It follows (e.g. \cite[Proposition 3.2, case (iii)]{tits_arbre}) that $\alpha$ must operate as a translation $x_i\xmapsto{\quad} x_{i+t}$, $i\in \ZZ$ of length $t$ along a unique bi-infinite oriented path
    \begin{equation*}
      \cdots \xrightarrow{\ f\ } x_{-1}
      \xrightarrow{\ f\ } x_0
      \xrightarrow{\ f\ } x_1
      \xrightarrow{\ f\ } \cdots.
    \end{equation*}
    As in the preceding case, the $G_i$ of \equref{eq:wrz} will then be the automorphism groups of the functional rooted trees associated to the restrictions $f|_{\Oo_{-,x_i}}$ to the backward orbits of $x_i$, $0\le i\le t-1$.

    On the other hand, suppose $\amp\equiv 0$ so that every automorphism of $f$ fixes {\it some} $x\in X$ (case (\ref{item:everyonefix}), in short). For any $x\in X$ we have
    \begin{equation*}
      \Aut(f) = \bigcup_{n\in \ZZ_{>0}}\textrm{Fix}(f^nx)
      \quad\textrm{where}\quad
      \textrm{Fix}(y):=\{\alpha\in \Aut(f)\ |\ \alpha(y)=y\}.
    \end{equation*}
    The removal from $\Gamma_f$ of the edges (but not the vertices) along the path
    \begin{equation}\eqlabel{eq:yfyfy}
      y\xrightarrow{\quad}fy\xrightarrow{\quad}f^2y\xrightarrow{\quad}\cdots
    \end{equation}
    leaves behind a countable union of directed rooted trees, with the vertices \equref{eq:yfyfy} as their respective roots. $\textrm{Fix}(y)$ is nothing but the product of the automorphism groups of said rooted trees, hence the conclusion.  \qedhere

  \end{enumerate}
\end{proof}

The justification given above for focusing on automorphism groups of {\it connected} functions was that those of arbitrary functions can be assembled from the former according to the scheme \equref{eq:arbfconnf}: products of wreath products. We will now strengthen that remark with the observation that the freedom in effecting that assemblage is absolute: given the (connected) $f_i$, there is some $f$ satisfying \equref{eq:arbfconnf} for {\it arbitrary} symmetric groups $\Sigma_{C_i}$.

\begin{theorem}\thlabel{th:arbprod}
  Let $G_i$, $i\in I$ be centralizers of connected functions and $\Sigma_i$ symmetric groups on respective sets $C_i$. The group
  \begin{equation}\eqlabel{eq:genwr}
    G:=\prod_{i\in I} G_i\wr \Sigma_i
  \end{equation}
  is then a function centralizer.
\end{theorem}

The idea will be to
\begin{itemize}
\item start with connected functional graphs $\Gamma_i$ with $G_i\cong \Aut(\Gamma_i)$;
\item alter them into new graphs $\Gamma'_i$ so as to ensure no two (for distinct $i\ne j\in I$) are isomorphic;
\item taking care not to disturb the individual automorphism groups in the process;
\item and then recover $G$ as the automorphism group of the disjoint-union graph
  \begin{equation*}
    \Gamma:=\coprod_{i\in I}\left(\textrm{$|C_i|$ copies of $\Gamma'_i$}\right).
  \end{equation*}
\end{itemize}

This will require that we formalize some of the constructions involved and introduce some additional terminology.

\begin{definition}\delabel{def:branch}
  A {\it branching pattern} is a sequence $\left(A_n\right)_{n\in \ZZ_{\ge 0}}$ of multisets of cardinal numbers, wherein the terms are recursively compatible on the following sense:
  \begin{itemize}
  \item $A_0=(\alpha_0)$ consists of a single cardinal number;
  \item the multiset
    \begin{equation*}
      A_1=(\alpha_{1,i}\ |\ i\in \alpha_0)
    \end{equation*}
     consists of $\alpha_0$ cardinal numbers;
   \item the members of the next multiset
     \begin{equation*}
      A_2=\left(\alpha_{2,j}\ \bigg|\ j\in \coprod_{i}\alpha_{1,i}\right)
    \end{equation*}
    are in bijection with the disjoint union of the cardinals which constitute the preceding multiset $A_1$, and so on.
  \end{itemize}
  For the purpose of referring to a generic $\alpha_{n,i}$, it will occasionally be convenient to write $\alpha_0$ as $\alpha_{0,0}$ (so the notation $\alpha_{n,i}$ covers all cardinals).
\end{definition}

To return to trees:

\begin{definition}\delabel{def:branchingtree}
  Let $A=(A_n)$ be a branching pattern in the sense of \deref{def:branch}.
  \begin{enumerate}[(1)]
  \item The directed rooted tree $\Gamma_A$ {\it associated to $A=(A_n)$} is constructed inductively as follows:
    \begin{itemize}
    \item the root has $\alpha_0$ incoming edges;
    \item the source of edge $i$, $i\in \alpha_0$ itself has $\alpha_{1,i}$ incoming edges;
    \item the source of edge $j$, $j\in \alpha_{1,i}$ has $\alpha_{2,j}$ incoming edges, etc.
    \end{itemize}
  \item\label{item:explos} For a directed graph $\Gamma$, the {\it $A$-explosion} $A\triangleright \Gamma$ of $\Gamma$ is obtained by attaching a copy of the directed rooted tree $\Gamma_A$ at every vertex of $\Gamma$, with that vertex as the root, and keeping all existing edges (in either the original $\Gamma$ or the various copies of $\Gamma_A$ being attached) unchanged.
  \end{enumerate}
\end{definition}

The $A$-explosion procedure now affords us a method to duplicate automorphism groups without accidentally producing isomorphic graphs.

\begin{proposition}\prlabel{pr:manysameauto}
  For any directed graph $\Gamma$, there are arbitrarily large sets of mutually non-isomorphic directed graphs, all with automorphism groups $\Aut(\Gamma)$.

  Furthermore, if $\Gamma$ is functional, or a functional tree, or a directed rooted tree then said graphs can also be chosen so as to have these properties, respectively.
\end{proposition}
\begin{proof}
  We tackle the functional case directly, as the construction will plainly be extensible to the general setup and preserve the properties of being a functional or directed rooted tree.

  Let $\Gamma$ be a functional digraph, and consider a branching pattern $A=(A_n)$ such that
  \begin{itemize}
  \item the cardinals $\alpha_{n,i}\in A_n$, for $n$ and $i$ varying freely, are all distinct;
  \item $\alpha_0$ is the largest of them;
  \item and furthermore, $\alpha_{n,i}$ are all larger than the in-degree of any vertex of $\Gamma$.
  \end{itemize}
  We claim, now, that the automorphism group of the $A$-explosion $A\triangleright \Gamma$ equals that of the original $\Gamma$.

  First, by construction (since in constructing $A\triangleright\Gamma$ we are gluing a copy of the same graph $\Gamma_A$ to every vertex of $\Gamma$), $\Aut(\Gamma)$ does act (faithfully) on $A\triangleright\Gamma$. We have, in other words, an obvious embedding
  \begin{equation*}
    \Aut(\Gamma)\le \Aut(A\triangleright\Gamma).
  \end{equation*}
  Conversely, we claim that every automorphism of $A\triangleright\Gamma$ arises by extending one of the original $\Gamma$. To see this, note that the edges in $\Gamma\subset A\triangleright \Gamma$ are precisely those whose sources have in-degree $\alpha_0$: all other edges belong to some copy of $\Gamma_A$ glued to $\Gamma$, and by assumption those edges' sources all have in-degrees
  \begin{equation*}
    \alpha_{n,i}<\alpha_0\textrm{ for some }n\in \ZZ_{>0}.
  \end{equation*}
  This means that $\Aut(A\triangleright\Gamma)$ leaves $\Gamma$ invariant, and hence every automorphism of $A\triangleright\Gamma$ must extend one of $\Gamma$. That extension is furthermore unique, because the fact that no two $\alpha_{n,i}$ are equal ensures that every copy of $\Gamma_A$ attached to $\Gamma$ is {\it uniquely} identifiable with any other copy.

  To conclude, note that we can now find arbitrarily large sets of branching patterns $A^{\kappa}$ satisfying the conditions above, so that no two $A^{\kappa}\triangleright\Gamma$ are mutually isomorphic: well-order the arbitrarily large set $\{\kappa\}$, and ask that for $\kappa<\kappa'$ every constituent cardinal $\alpha^k_{n,i}$ of $A^{\kappa}$ be smaller than every constituent cardinal $\alpha^{\kappa'}_{m,j}$ of $A^{\kappa'}$.
\end{proof}

\begin{proof}[Proof of \thref{th:arbprod}]
  We can now pursue the plan outlined after the statement of the theorem, using \prref{pr:manysameauto} as a tool: write each $G_i$ as the centralizer of a function $f_i:X_i\to X_i$ with $f_i$ not isomorphic to $f_j$ for $i\ne j$; this is possible even if $G_i$ and $G_j$ are isomorphic, by \prref{pr:manysameauto}.

  Next, consider the function $f:X\to X$ obtained by gluing together copies of the $f_i$ in the obvious fashion, where
  \begin{equation*}
    X:=\coprod_i X_i^{\coprod C_i}
  \end{equation*}
  and $Y^{\coprod Z}$ denotes the disjoint union of $|Z|$ copies of $Y$.
\end{proof}

In the same spirit (as in \thref{th:arbprod}), and using the same methods, we can show that the groups of the form \equref{eq:wrzs} or \equref{eq:wrz} are {\it precisely} those arising as in parts (\ref{item:cycle}) and (\ref{item:nofix}) of \thref{th:autogp-types} respectively.

\begin{theorem}\thlabel{th:allwreathcyclic}
  Let $G_i$, $1\le i\le t$ be a finite family of directed-rooted-tree automorphism groups.
  \begin{enumerate}[(1)]

  \item\label{item:anywrzs} The wreath product
    \begin{equation*}
      \left(\prod_{i=1}^t G_i\right)\wr \ZZ/s
      \quad \textrm{with $\ZZ/s$ acting cyclically on a size-$s$ set}\quad
    \end{equation*}
    is a function centralizer, as in \thref{th:autogp-types} (\ref{item:cycle}).

  \item\label{item:anywrz} The wreath product
    \begin{equation*}
      \left(\prod_{i=1}^t G_i\right)\wr \ZZ
      \quad \textrm{with $\ZZ$ acting on itself by translations}\quad
    \end{equation*}
    is a function centralizer, as in \thref{th:autogp-types} (\ref{item:nofix}).

  \end{enumerate}
\end{theorem}
\begin{proof}
  The differences between the two cases are not substantial, so we focus on part (\ref{item:anywrz}).

  Start with a bi-infinite directed path
  \begin{equation*}
    P:=\cdots \xrightarrow{\quad}x_{-1} \xrightarrow{\quad} x_0 \xrightarrow{\quad} x_1\xrightarrow{\quad}\cdots
  \end{equation*}
  with $\ZZ$-indexed vertices. To construct the desired functional digraph $\Gamma$, attach a directed rooted tree $T_i$, $1\le i\le t$ with automorphism group $G_i$ to each
  \begin{equation*}
    x_j,\quad j=i-1\mod t
  \end{equation*}
  with $x_j$ as the root. Said trees can also be enhanced to $A_i$-explosions thereof in the sense of \deref{def:branchingtree} (\ref{item:explos}), as in the proof of \prref{pr:manysameauto}, so as to ensure that
  \begin{equation*}
    \Aut(\Gamma)\cong \left(\prod_{i=1}^t G_i\right)\wr \ZZ.
  \end{equation*}
  One way to achieve this:
  \begin{itemize}
  \item Choose the branching patterns $A_i$ so as to ensure that the edges of the original path $P$ are the only ones in $\Gamma$ whose sources receive more than $\kappa$ edges for some large cardinal number $\kappa$. This will force $\Aut(\Gamma)$ to leave the path $P$ invariant, and hence also map the $T_i$ onto one another.

  \item Make sure the starting cardinals $\alpha_{i,0}$ respectively attached to the branching patterns $A_{i}$, $1\le i\le t$ are all distinct, thus further forcing $\Aut(\Gamma)$ to only map the trees attached to $x_j$ to those attached to $x_{j+kt}$, $k\in \ZZ$.
  \end{itemize}
  It should be clear now how to replicate this procedure for the finite cyclic case (\ref{item:anywrzs}).
\end{proof}

This leaves the situation of \thref{th:autogp-types} (\ref{item:everyonefix}) somewhat in the air, but there too we can be more explicit. For one thing, the inclusions making up the tower of products in that statement are of a very special nature.

First, it should be clear from \thref{th:arbprod} that products of the form \equref{eq:genwr} appear quite frequently and naturally. It will be convenient, for that reason, to have some compact notation for them.

\begin{definition}\delabel{def:mulwr}
  A {\it multi-wreath product} is a product of the form \equref{eq:genwr} for groups $G_i$ and symmetric groups $\Sigma_i:=\Sigma_{C_i}$ attached to sets $C_i$ respectively. We will denote such products by
  \begin{equation*}
    (G_i)\wr (C_{i}):=\prod_{i\in I}G_i\wr \Sigma_{C_i}.
  \end{equation*}
  When the family $I$ is a singleton we drop the parentheses, writing simply $G\wr C$ for $G\wr \Sigma_C$.
\end{definition}

Having set up that piece of notation, we also need

\begin{definition}\delabel{def:wrmerge}
  A {\it wreath merger} is a group embedding of the form
  \begin{equation}\eqlabel{eq:wrmerge}
    G\times G\wr C\times (G_i)\wr (C_i)
    \xrightarrow{\quad}
    G\wr (C\sqcup\{*\}) \times (G_i)\wr (C_i)
  \end{equation}
  for some groups $G$, $G_i$, $i\in I$ and sets $C$ and $C_i$, $i\in I$, where
  \begin{itemize}
  \item the map
    \begin{equation*}
      G\times G\wr C\to G\wr (C\sqcup\{*\})
    \end{equation*}
    is the obvious embedding;
  \item and the morphism acts as the identity on the other factors.
  \end{itemize}
  In other words, such a morphism aggregates one factor $G$ to another factor $G\wr C$ into the larger wreath product over the larger set $C\sqcup{*}$, and leaves the other factors undisturbed.
\end{definition}

\begin{remark}
  The set $C$ of \deref{def:wrmerge} can also be empty, in which case the wreath merger is nothing but the obvious identification of the two sides.
\end{remark}

To return to \thref{th:autogp-types}:

\begin{lemma}\lelabel{le:iswrmerge}
  The embeddings that constitute the union of \thref{th:autogp-types} (\ref{item:everyonefix}) are wreath mergers in the sense of \deref{def:wrmerge}, with the groups $G$ and $G_i$ all being directed-rooted-tree automorphism groups.
\end{lemma}
\begin{proof}
  This is virtually immediate, once one unpacks the proof of \thref{th:autogp-types} (\ref{item:everyonefix}).

  For a vertex $y\in X$ and some other vertex $f^ny$, $n\ge 0$ along the path \equref{eq:yfyfy} denote by $T_{y,n}$ the directed tree with root $f^ny$ left over upon removing the edges of \equref{eq:yfyfy} (but not its vertices). The proof of \thref{th:autogp-types} (\ref{item:everyonefix}), rephrased using this notation, says that $G$ is the union of
  \begin{equation*}
    \prod_{n\ge 0}\Aut\left(T_{y,n}\right)
    \xrightarrow{\quad}
    \prod_{n\ge 0}\Aut\left(T_{fy,n}\right)
    \xrightarrow{\quad}
    \prod_{n\ge 0}\Aut\left(T_{f^2y,n}\right)
    \xrightarrow{\quad}
    \cdots
  \end{equation*}
  We focus on the first of these inclusions, as the discussion applies to the others with minimal notational changes.

  Note that for $n\ge 2$ we have $T_{y,n}=T_{fy,n-1}$, so the first of these inclusions is
  \begin{equation*}
    \Aut\left(T_{y,0}\right)\times \Aut\left(T_{y,1}\right)\times \left(\Aut\left(T_{y,2}\right)\times\cdots\right)
    \xrightarrow{\quad}
    \Aut\left(T_{fy,0}\right) \times \left(\Aut\left(T_{y,2}\right)\times\cdots\right),
  \end{equation*}
  with the parenthesized factors identical and the morphism acting identically on those factors.

  On the other hand, the tree $T_{fy,0}$ consists of $T_{y,1}$, with the same root $fy$, and $T_{y,0}$ attached to that root via the edge $y\to fy$:
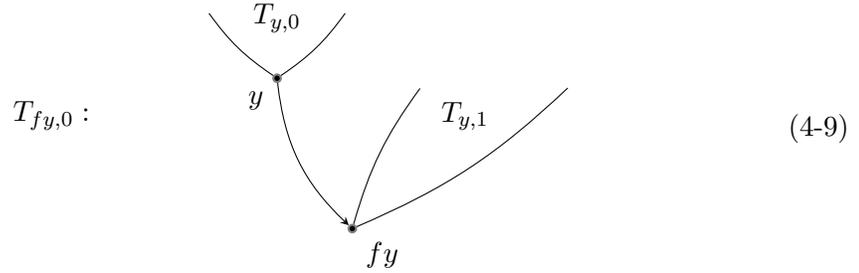
  \begin{equation}\eqlabel{eq:tytree0}
    \begin{tikzpicture}[>=stealth,auto,baseline=(current  bounding  box.center),place/.style={circle,draw=black!50,fill=black!100,thick,inner sep=0pt,minimum size=1mm}]

      \node at ( 0,0) [place] (y) [label=south west:$y$] {};
      \node at (1,-2) [place] (fy) [label=south east:$fy$] {};

      \node at (-1,1) (yl) {};
      \node at (1,1) (yr) {};

      \node at (2,0) (fyl) {};
      \node at (4,0) (fyr) {};

      \node at (0,.8) (ty0) {$T_{y,0}$};
      \node at (2.5,-.5) (ty1) {$T_{y,1}$};

      \node at (-3,-.5) (tfy0) {$T_{fy,0}:$};

      \draw [->] (y) to[bend right=20] (fy);
      \draw [-] (yl) to[bend right=10] (y);
      \draw [-] (yr) to[bend left=10] (y);

      \draw [-] (fyl) to[bend right=10] (fy);
      \draw [-] (fyr) to[bend left=10] (fy);
    \end{tikzpicture}
  \end{equation}
  $\Aut\left(T_{y,1}\right)$ is a multi-wreath product of automorphism groups of trees obtained by removing the edges incident to the root $fy$ in $T_{y,1}$, and $\Aut\left(T_{fy,0}\right)$ may or may not map the tree $T_{y,0}\cup(y\to fy)$ onto other sub-trees of $T_{fy,0}$ attached to the root $fy$. In any event, it should be clear now that
  \begin{equation*}
    \Aut\left(T_{y,0}\right)\times \Aut\left(T_{y,1}\right)
    \xrightarrow{\quad}
    \Aut\left(T_{fy,0}\right)
  \end{equation*}
  is a map of the form \equref{eq:wrmerge}, with
  \begin{itemize}
  \item $G=\Aut\left(T_{y,0}\right)$;
  \item $C$ being the set (perhaps empty) of sub-trees of $T_{fy,0}$ onto which
    \begin{equation*}
      T_{y,0}\cup(y\to fy)\subset T_{fy,0}
    \end{equation*}
    is mapped by elements of $\Aut\left(T_{fy,0}\right)$;
  \item and $G_i$ are automorphism groups of other sub-trees of $T_{fy,0}$, also rooted at $fy$, each with $fy$ having in-degree 1 therein.  \qedhere
  \end{itemize}
\end{proof}

There is a converse to \leref{le:iswrmerge}, characterizing the groups arising in \thref{th:autogp-types} (\ref{item:everyonefix}):

\begin{theorem}\thlabel{th:alleveryonefix}
  The groups realizable as in \thref{th:autogp-types} (\ref{item:everyonefix}) are precisely the embedding-chain unions of the form
  \begin{equation}\eqlabel{eq:chainofmergers}
    G_0\times H_1\times H_2\times H_3\times\cdots
    \xrightarrow{\quad}
    G_1\times H_2\times H_3\times\cdots
    \xrightarrow{\quad}
    G_2\times H_3\times\cdots
  \end{equation}
  where
  \begin{itemize}
  \item all $G_i$ and $H_i$ are directed-rooted-tree automorphism groups;
  \item and the $n^{th}$ embedding in the chain (starting the count at $n=0$) is the identity on the factors $H_k$, $k\ge n+2$ and restricts to a wreath merger
    \begin{equation*}
      G_n\times H_{n+1}
      \xrightarrow{\quad}
      G_{n+1}
    \end{equation*}
    in the sense of \deref{def:wrmerge}.
  \end{itemize}
\end{theorem}
\begin{proof}
  \leref{le:iswrmerge} shows that the groups of \thref{th:autogp-types} (\ref{item:everyonefix}) are indeed of this form, so it remains to argue the converse: that every such chain is realizable as $\Aut(\Gamma)$ for a functional tree $\Gamma=\Gamma_f$ (not rooted!) as in \thref{th:autogp-types} (\ref{item:everyonefix}).

  We thus assume the chain \equref{eq:chainofmergers} given. To construct the desired functional tree $\Gamma$, we will reverse the procedure in the proof of \leref{le:iswrmerge}, retaining also the notation $T_{y,n}$ for the trees we are working with.

  Starting with a path \equref{eq:yfyfy}, we will want to attach to it directed rooted trees $T_{y,n}$, with the $f^ny$ as their respective roots, so that
  \begin{equation*}
    \begin{aligned}
      G_n&\cong \Aut\left(T_{f^ny,0}\right),\ \forall n\ge 0,\\
      H_n&\cong \Aut\left(T_{y,n}\right) \cong \Aut\left(T_{f^{n-1}y,1}\right),\ \forall n\ge 1
    \end{aligned}
  \end{equation*}
  and the wreath merger
  \begin{equation*}
    \Aut\left(T_{f^ny,0}\right)\times \Aut\left(T_{y,n+1}\right)
    \cong
    G_n\times H_{n+1}\to G_{n+1}
    \cong
    \Aut\left(T_{f^{n+1}y,0}\right)
  \end{equation*}
  is the obvious one, resulting from the decomposition \equref{eq:tytree0} applied to $f^n y$ in place of $y$, taking into account the identification
  \begin{equation*}
    T_{y,n+1}=T_{f^ny,1},\ \forall n\ge 0.
  \end{equation*}
  The substance of the construction consists in choosing the $T_{y,n}$ judiciously, subject to these constraints. Regardless of how we do that, though, we will also have (and use) a map $|\cdot|:\Gamma\to \ZZ$ as in the proof of \thref{th:autogp-types}:
  \begin{equation*}
    |z|:=n-\left(\textrm{length of path from $z$ to the root $f^ny$ of $T_{y,n}$}\right),\ \forall z\in T_{y,n}.
  \end{equation*}
  To begin with, $G_0$ is by assumption the automorphism group of a directed rooted tree. We can take $T_{y,0}$ to be an $A$-explosion of such a tree for some branching pattern $A$ (Definitions \ref{de:def:branch} and \ref{de:def:branchingtree}) whose constituent cardinals are large (how large, we estimate below) and distinct, as in the proof of \prref{pr:manysameauto}.

  If the initial cardinal $\alpha_0$ of $A$ is the largest, and also larger than all of the cardinals entering the wreath products implicit in \equref{eq:chainofmergers}, then the only vertices $z$ of the resulting tree $\Gamma$ (constructed according to the recipe sketched below) with in-degree $\ge \alpha_0$ satisfy $|z|=0$, so that indeed all elements of $\Aut(\Gamma)$ will fix some $f^n y$.

  Next, consider the wreath merger
  \begin{equation*}
    G_0\times H_1\to G_1,
  \end{equation*}
  expressed in the form \equref{eq:wrmerge}. The next tree $T_{y,1}$ that we have to select will be one with automorphism group $H_1$, consisting of
  \begin{itemize}
  \item $|C|$ copies of the preceding tree $T_{y,0}$ attached with $|C|$ edges to the root $fy$;
  \item and for each $i$ (notation as in \equref{eq:wrmerge}) $|C_i|$ copies of the same tree with automorphism group $G_i$;
  \item and with these trees exploded along branching patterns with distinct large cardinals, but all smaller than those of the initial branching pattern $A$.
  \end{itemize}
  Now repeat the procedure: at each step, when constructing $T_{y,n}$, one must include some copies of the preceding tree $T_{f^{n-1}y,0}$ to account for the merger, along with other trees exploded along branching patterns whose cardinals are {\it smaller} than those of the preceding branching patterns employed.

  There is of course an ample supply of cardinals for all of this to go through, hence the conclusion.
\end{proof}

The joint effect of Theorems \ref{th:th:autogp-types}, \ref{th:th:arbprod}, \ref{th:th:allwreathcyclic} and \ref{th:th:alleveryonefix} is that of characterizing the groups of the form $\Aut(f)$ in terms of automorphism groups of directed rooted trees. These, in turn, are easily described.

To that end, recall \cite[\S 1]{bond_itwr} that for a sequence $(G_n)_{n\ge 0}$ of permutation groups (acting, say, on the sets $X_n$) one can define the {\it iterated wreath product}
\begin{equation*}
  \cdots\wr G_1\wr G_0 = \varprojlim_n G_n\wr \cdots\wr G_1\wr G_0,
\end{equation*}
the right-hand side being defined inductively as the usual wreath product
\begin{equation*}
  G_n\wr \left(G_{n-1}\wr \cdots\wr G_0\right)
\end{equation*}
associated to the (only reasonable) action of the previous iteration $G_{n-1}\wr \cdots\wr G_0$ on $X_{n-1}\times\cdots\times X_0$.

\begin{theorem}\thlabel{th:itwr}
  The groups realizable as automorphism groups of directed rooted trees are precisely the iterated wreath products $\cdots\wr G_1\wr G_0$, where each $G_n$ is a permutation group of the form
  \begin{equation*}
    \prod_{j\in S_n}\Sigma_{X_{nj}}\textrm{ acting on the disjoint union }\coprod_{j\in S_n}X_{nj}
  \end{equation*}
  for sets $S_n$ and $X_{nj}$, $j\in S_n$.
\end{theorem}
\begin{proof}
  Consider, first, a directed rooted tree $\Gamma$, with root $x_0$ and edges $e_i$, $i\in I$ incident to it. The automorphism group $G:=\Aut(\Gamma)$ permutes the directed rooted sub-trees $\Gamma_i\subset\Gamma$ with the sources of $e_i$ as their respective roots, partitioning that set of trees into orbits, as
  \begin{equation*}
    I = \coprod_j X_{0j},\quad j\in\textrm{ set $S_0$ of orbits}.
  \end{equation*}
  This gives surjection
  \begin{equation*}
    G\to \prod_{j\in S_0}\Sigma_{X_{0j}}=:G_0,
  \end{equation*}
  easily seen to be split, with kernel $\prod\Aut(\Gamma_i)$. It should be clear now how, repeating this procedure with the $\Gamma_i$ in place of $\Gamma$, one can recover $G$ as the desired inverse limit.

  Conversely, consider an iterated wreath product $G$ as in the statement; we have to argue that a directed rooted tree $\Gamma$ exists, with $G$ as its automorphism group. This requires that
  \begin{itemize}
  \item we have edges adjacent to the root $x_0$ indexed by the disjoint union $\coprod_{j\in S_0}X_{0j}$;
  \item and that for every $j\in S_0$ the edges in $X_{0j}$ attach to $x_0$ mutually isomorphic trees $\Gamma_x$, $x\in X_{0,j}$;
  \item which are not isomorphic to those associated to other $j'\in S_0$, $j'\ne j$;
  \item any {\it other} edges incident to the root $x_0$, apart from those in $\coprod_{j\in S_0}X_{0j}$, attach to $x_0$ trees that are both mutually non-isomorphic and non-isomorphic to any of the $\Gamma_x$.
  \item and so on, with the $\Gamma_x$ in place of $\Gamma$.
  \end{itemize}
  That such a pattern of (non-)isomorphism is achievable should be easy to see by now, via the same process (used repeatedly above) of attaching trees $\Gamma_A$ associated to branching patterns $A$: to ensure that $\Gamma_x$, $x\in X_{0j}$ and $\Gamma_x'$, $x'\in X_{0j'}$ are non-isomorphic for $j\ne j'$, for instance, one can attach $\Gamma_{A_j}$ and $\Gamma_{A_{j'}}$ to the roots of all
  \begin{equation*}
    \Gamma_x,\ x\in X_{0j}
    \quad\textrm{and}\quad
    \Gamma_x',\ x\in X_{0j'}
    \quad\textrm{respectively},
  \end{equation*}
  ensuring that the initial cardinals of $A_j$ and $A_{j'}$ are distinct for $j\ne j'$.
\end{proof}

Piecing together the various results above will now give a complete characterization of the groups of the form $\Aut(f)\cong \Aut(X,R_f)$. As consequence, the abelian ones are easily described:

\begin{corollary}\colabel{cor:abautf}
  For an abelian group $G$, the following conditions are equivalent:
  \begin{enumerate}[(a)]

  \item\label{item:isautf} $G$ is of the form $\Aut(f)$ for some $f$.

  \item\label{item:isautfab} $G$ is an abelianization $\Aut(f)_{ab}$ for some $f$.

  \item\label{item:prodcyc} $G$ is a product of cyclic groups.

  \end{enumerate}
\end{corollary}
\begin{proof}
  The implication (\ref{item:isautf}) $\Longrightarrow$ (\ref{item:isautfab}) is formal, so we focus on the others.

  \begin{enumerate}[]

  \item {\bf (\ref{item:prodcyc}) $\Longrightarrow$ (\ref{item:isautf})} A cyclic group $G$ generated by $\sigma$ is the automorphism group of the function
    \begin{equation*}
      G\xrightarrow{\quad\textrm{translation by $\sigma$\quad}}G,
    \end{equation*}
    and products of such are then function centralizers by \thref{th:arbprod}.

  \item {\bf (\ref{item:isautfab}) $\Longrightarrow$ (\ref{item:prodcyc})} This follows from \prref{pr:autgpdec}, \thref{th:autogp-types} and \thref{th:itwr} after observing that the abelianization of a symmetric group is either trivial (if the symmetric group is infinite) or $\ZZ/2$.  \qedhere

  \end{enumerate}
\end{proof}

$\fse$ solutions $(X,R)$ in general allow for more variety to their automorphism groups than single functions (i.e. than the solutions $R_f$ of \exref{exemFS} (\ref{item:fh})). Even Kimura semigroups alone (a poorer structure than an $\fse$ solution by \thref{refqm}) exhibit much flexibility in this regard:

\begin{theorem}\thlabel{th:allgpsquotkimura}
  For any group $G$, there is a Kimura semigroup $(X,\cdot)$, a permutation action of $G$ on some set $C$ and permutation groups $\Sigma_c$, $c\in C$ so that we have a semidirect product decomposition
  \begin{equation*}
    \Aut(X,\cdot)\cong \left(\prod_c\Sigma_c\right)\rtimes G
  \end{equation*}
  for the action of $G$ permuting $c\in C$.
\end{theorem}

Before going into the proof, a consequence:

\begin{corollary}\colabel{cor:notallrf}
  There are automorphism groups of Kimura semigroups that are not function centralizers. Consequently, the class of $\fse$-solution automorphism groups is strictly broader than that of function centralizers.
\end{corollary}
\begin{proof}
  The second claim follows from the first, given that A Kimura semigroup provides an $\fse$ solution as in \thref{refqm} when supplemented with the identity quasi-endomorphism.

  As for the first claim, take $G=(\QQ,+)$ in \thref{th:allgpsquotkimura}. The theorem says that some $\Aut(X,\cdot)$ admits $G$ as a split quotient, but this is not so for groups of the form $\Aut(f)$.

  To see this, note first that a split quotient in $\Aut(f)$ would be one in the abelianization $\Aut(f)_{ab}$, which in turn is a product of cyclic groups by \coref{cor:abautf}. That $(\QQ,+)$ is not a split quotient thereof follows, for instance, from the fact that $(\QQ,+)$ is {\it divisible} \cite[p.320, Definition following Exercise 10.17]{rot_gp} whereas products of cyclic groups are easily seen not to be.
\end{proof}

\begin{proof}[Proof of \thref{th:allgpsquotkimura}]
  In handling Kimura semigroups, we will use the language of \thref{th:refqmbis}. According to it (or rather a paraphrase thereof), such a structure consists of a set decomposed as a direct product $A\times B$ together with a surjective map
  \begin{equation*}
    k:S\to A\times B.
  \end{equation*}
  Indeed, this is a repackaging of the idempotent function $h:X\to X$ in the statement of \thref{th:refqmbis} via
  \begin{equation*}
    A\times B = \im~h,\quad S=X\setminus \im~h,\quad k=h|_S.
  \end{equation*}
  In this framework, \thref{th:refqmbis} (\ref{item:automeansthis}) then describes $\Aut(X,\cdot)$ as the group of invertible intertwiners of $k$, splitting as a product of separate permutations on $A$ and $B$:
  \begin{equation}\eqlabel{eq:autxdot}
    \Aut(X,\cdot)\cong \{(\sigma_S,\ \sigma_A,\ \sigma_B)\ |\ \sigma_{\bullet}\in \Sigma_{\bullet}\textrm{ for $\bullet=S,A,B$ and }(\sigma_A\times\sigma_B)\circ k = k\circ\sigma_S\}.
  \end{equation}
  Given $\sigma_A\in \Sigma_A$ and $\sigma_B\in \Sigma_B$, a permutation $\sigma_S$ fitting into \equref{eq:autxdot} exists if and only if
  \begin{equation}\eqlabel{eq:samecardab}
    |k^{-1}(a, b)| = |k^{-1}(\sigma_A a,\sigma_B b)|,\ \forall (a,b)\in A\times B.
  \end{equation}
  Consider, now, the group $G\le \Sigma_A\times\Sigma_B$ consisting of pairs $(\sigma_A,\sigma_B)$ satisfying \equref{eq:samecardab}. Casting the cardinal numbers $|k^{-1}(a,b)|$ as colors, the assignment
  \begin{equation*}
    A\times B\ni (a,b)\xmapsto{\quad} c_{a,b}:=|k^{-1}(a,b)|
  \end{equation*}
  can be thought of as a {\it bipartite} \cite[\S 1.6]{Die00} edge-colored graph with $(A,B)$ as its bipartition. Fixing bijections
  \begin{equation*}
    k^{-1}(a,b)\cong k^{-1}(a',b')\textrm{ whenever }c_{a,b}=c_{a',b'}
  \end{equation*}
  will then give a decomposition
  \begin{equation*}
    \Aut(X,\cdot)\cong \left(\prod_{a,b}\Sigma_{c_{a,b}}\right)\rtimes G,
  \end{equation*}
  where for a cardinal $\alpha$ we write $\Sigma_{\alpha}$ for the permutation group of a set of that cardinality. We will be done, then, taking $C:=A\times B$, once we argue that every group $G$ arises as the automorphism group of an edge-colored bipartite graph (where automorphisms of such a structure are understood to preserve colors). We relegate this to \leref{le:coloredbipart}, which will finish the present proof.
\end{proof}

\begin{lemma}\lelabel{le:coloredbipart}
  Every group $G$ is the automorphism group of some bipartite edge-colored graph.
\end{lemma}
\begin{proof}
  Given a set $S$ of monoid generators for $G$, recall \cite[p.357, Definition preceding Exercise 11.13]{rot_gp} that the {\it Cayley graph} $\Gamma(G,S)$ of $G$ with respect to $S$ is the directed graph with $G$ itself as its vertex set and an edge
  \begin{equation*}
    g\xrightarrow{\quad}gs,\quad \forall s\in S,\ g\in G.
  \end{equation*}
  Assigning $s\in S$ itself as a color, one can regard $\Gamma(G,S)$ as an {\it edge-colored} directed graph. Its automorphism group as such is then nothing but $G$: a bijection $\theta:G\to G$ preserves the (non-)existence of edges along with their colors if and only if
  \begin{equation*}
    \theta(gs) = \theta(g)s,\quad \forall s\in S,\ g\in G,
  \end{equation*}
  so
  \begin{align*}
    \theta(g) = g_0 g&\Longrightarrow \theta(gs) = g_0gs,\quad\forall s\in S,\ g\in G\\
                     &\Longrightarrow \theta(g)=g_0g,\quad\forall g\in G
                       \quad\textrm{because $S$ generates $G$ as a monoid.}
  \end{align*}
  It now remains to observe that the edge-colored directed graph $\Gamma(G,S)$ can be repackaged as an edge-colored bipartite ({\it un}directed) graph $\Gamma$:
  \begin{itemize}
  \item the bipartition will be $(G,S)$;
  \item for every $s\in S$ we introduce two distinct symbols (colors) $s_i$, $i=0,1$;
  \item and for every oriented edge
    \begin{equation*}
      g\xrightarrow{\quad s\quad}gs
    \end{equation*}
    in the original graph $\Gamma(G,S)$ there is an $s_0$-colored edge in $\Gamma$ connecting $g\in G$ and $s\in S$, and an $s_1$-colored edge connecting $s\in S$ and $gs\in G$.
  \end{itemize}
  That the automorphism group of $\Gamma$ as an edge-colored graph equals that of $\Gamma(G,S)$ (and is thus $G$) is a simple exercise.
\end{proof}

\subsection*{Supporting Data}
No datasets were generated or analyzed during the current study.

\subsection*{Competing Interests}
The authors have no competing interests to declare that are relevant to the content of this article.



\end{document}